\documentclass[12pt,reqno]{amsart}
\usepackage{amssymb,amscd,amsmath,amsthm,color}
\newtheorem{thm}{Theorem}[section]
\newtheorem{prop}[thm]{Proposition}
\newtheorem{lemma}[thm]{Lemma}
\newtheorem{cor}[thm]{Corollary}
\newtheorem{definition}[thm]{Definition}
\newtheorem{remark}[thm]{Remark}

\newtheorem{problem}[thm]{Problem}
\numberwithin{equation}{section}

\def\bR{\mathbb{R}}

\def\bN{\mathbb{N}}
\def\cH{\mathcal{H}}

\def\bP{\mathbb{P}}
\def\tr{\mathrm{tr}}
\def\<{\langle}
\def\>{\rangle}
\def\dT{d_{\mathrm T}}
\def\eps{\varepsilon}

\def\bw{\mathbf{w}}
\def\cA{\mathcal{A}}
\def\cB{\mathcal{B}}
\def\cP{\mathcal{P}}
\def\cU{\mathcal{U}}
\def\cL{\mathcal{L}}

\def\ffi{\varphi}
\def\cPP{\cP^\infty(\bP)}
\def\fM{\mathfrak{M}}
\def\cK{\mathcal{K}}
\def\cN{\mathcal{N}}
\def\cO{\mathcal{O}}
\def\supp{\mathrm{supp}}

\begin{document}
\baselineskip=15pt
\allowdisplaybreaks

\title[Operator means of probability measures]
{Operator means of probability measures}
\author[F.~Hiai]{Fumio Hiai}
\address{Graduate School of Information Sciences, Tohoku University,
Aoba-ku, Sendai 980-8579, Japan}
\email{hiai.fumio@gmail.com}
\author[Y.~Lim]{Yongdo Lim}
\address{Department of Mathematics, Sungkyunkwan University, Suwon 440-746, Korea}
\email{ylim@skku.edu}
\date{\today}
\maketitle

\begin{abstract}
Let $\bP$ be the complete metric space consisting of positive
invertible operators on an infinite-dimensional Hilbert space with
the Thompson metric. We introduce the notion of operator means of
probability measures on $\bP$, in parallel with Kubo and Ando's
definition of two-variable operator means, and show that every
operator mean is contractive for the $\infty$-Wasserstein distance.
By means of a fixed point method we consider deformation of such
operator means, and show that the deformation of any operator mean
becomes again an operator mean in our sense. Based on this
deformation procedure we prove a number of properties and
inequalities for operator means of probability measures.


\end{abstract}

\section{Introduction}

A systematic study of two-variable operator means of positive operators on a Hilbert space
$\cH$ began with the paper of Kubo and Ando \cite{KA}. There is one-to-one correspondence
between the operator means $\sigma$ in the sense of Kubo-Ando and the positive operator
monotone functions $f$ on $[0,\infty)$ with $f(1)=1$ in such a way that
$A\sigma B=a^{1/2}f(A^{-1/2}BA^{-1/2})A^{1/2}$ for positive invertible operators $A,B$ on
$\cH$. The geometric mean, first introduced by Pusz and Woronowicz \cite{PW} and then
discussed by Ando \cite{An} in detail, is a two-variable operator mean that has been paid
the most attention. It was a long-standing problem to extend the notion of the geometric
mean to the case of more than two variables of matrices/operators. A breakthrough came when
the definitions of multivariate geometric means of positive definite matrices appeared
in the iteration approach by Ando, Li and Mathias \cite{ALM} and in the Riemannian geometry
approach by Moakher \cite{Mo} and by Bhatia and Holbrook \cite{BhHo2}. Since then, the
Riemannian multivariate operator means, in particular, the Karcher mean (the
generalization of the geometric mean) and the power means, have extensively been developed
by many authors, see among others \cite{LL1,BK,LP}. Furthermore, those multivariate
operator means have recently been generalized to probability measures on the positive
definite matrices in connection with the Wasserstein distance, see, e.g.,
\cite{KL,LL2,HL,HL2}.

We write $\bP$ for the set of positive invertible operators on the Hilbert space $\cH$. An
imortant point in the Riemannian geometry approach to operator means when $\dim\cH<\infty$ is
that $\bP$ forms a Riemannian manifold with non-positive curvature (referred to as a global
NPC space \cite{St}). Even when $\dim\cH=\infty$, $\bP$ is a Banach-Finsler manifold with the
Thompson metric, although it can no longer have a Riemannian manifold structure. Thus, we can
study operator means of probability measures on $\bP$ in connection with theory of contractive
barycenters with respect to the Wasserstein distance and related stochastic analysis (e.g.,
ergodic theorems), developed in the framework of complete metric spaces in, e.g.,
\cite{EH,St02,St,HL}. Moreover, operator means of probability measures, in turn, provide good
examples of contractive barycenters.

In recent study of operator means, the fixed point method, apart
from the Riemannian geometry method, provides a main technical tool
as used in different places in, e.g.,
\cite{KL,KLL2,LL0,LL,LP,LP3,LP2,Pa,Ya}. In particular, the Karcher
and the power means are defined as the solutions to certain fixed
point type equations. In this status of the subject matter, our aim
of the present paper is to systematically develop the fixed point
method for operator means of Borel probability measures on $\bP$
with bounded support. In our approach, we apply the fixed point
method based on monotone convergence of Borel probability measures
in terms of the strong operator topology, where the stochastic order of
probability measures discussed in \cite{HLL} plays a key role. The idea
using monotone convergence is essentially in a similar vein to that of
Kubo-Ando's definition of two-variable operator means. In previous studies
of the subject in the fixed point method, a primary tool is the Banach
contraction principle, which we never use in the present paper.
Thus, the class of operator means of multivariables and probability
measures studied in the paper is considerably wider than those in
other papers so far.

The paper is organized as follows. We write $\cPP$ for the set of Borel probability measures
on $\bP$ with bounded support (of full measure). In Section 2 we first fix the notion of
monotone convergence for a sequence of probability measures in $\cPP$ (see Definition
\ref{D-2.3}), which plays a primary role in our study as mentioned above. We then give the
definition of operator means (see Definition \ref{D-2.5})
$$
M:\cPP\ \longrightarrow\ \bP
$$
in parallel with Kubo-Ando's definition of two-variable operator means, where one important
requirement is the monotone continuity that if $\mu_k\nearrow\mu$ or $\mu_k\searrow\mu$ in
$\cPP$, then $M(\mu_k)\to M(\mu)$ in the strong operator topology. It is also shown here
that every operator means on $\cPP$ automatically has the contractivity for the
$\infty$-Wasserstein distance with respect to the Thompson metric.

In Section 3 we present the main theorem (Theorem \ref{T-3.1}) of the paper. For an
operator mean $M$ on $\cPP$ and a two-variable operator mean $\sigma$ ($\ne$ the left
trivial mean) and for any $\mu\in\cPP$, the theorem says that the fixed point type equation
$$
X=M(X\sigma\mu)\qquad\mbox{for $X\in\bP$},
$$
where $X\sigma\mu$ is the push-forward of $\mu$ by the map $A\in\bP\mapsto X\sigma A\in\bP$,
has a unique solution, and if we denote the solution by $M_\sigma(\mu)$, then $M_\sigma$
is an operator mean on $\cPP$ again. We call $M_\sigma$ the deformed operator mean from
$M$ by $\sigma$. The notion of deformed operator means is considered in some sense as an
extended version of the generalized operator means by P\'alfia [25] (see Remark \ref{R-3.2}).
The deformation procedure $M\to M_\sigma$ has the order property that
\begin{align}\label{F-1.1}
X\le M(X\sigma\mu)\,\implies\,X\le M_\sigma(\mu),\quad
X\ge M(X\sigma\mu)\,\implies\,X\ge M_\sigma(\mu).
\end{align}
This property is quite useful in later discussions. In Section 4 we prove that several
basic properties such as congruence invariance and concavity are preserved under the
procedure $M\to M_\sigma$.

In Section 5 we show that all of the arithmetic, the harmonic, the
Karcher, and the power means are operator means on $\cPP$ in our
sense. By starting from those familiar means and by taking deformed
operator means again and again, we have a rich class of operator
means on $\cPP$. By applying the property in \eqref{F-1.1} we can
show many inequalities for operator means on $\cPP$. For instance,
in Section 6, the inequality under positive linear maps and
Ando-Hiai type inequalities are obtained for certain classes of
operator means on $\cPP$. Furthermore in Section 7, when $\cH$ is
finite-dimensional, a certain norm inequality, eigenvalue
majorizations and the Minkowski determinant inequality are obtained.

\section{Definitions}

Let $\cH$ be a \emph{separable} Hilbert space and $B(\cH)$ be the algebra of all bounded
linear operators on $\cH$. Let $B(\cH)^+$ be the set of positive (not necessarily invertible)
operators in $B(\cH)$, and $\bP=\bP(\cH)$ be the set of positive invertible operators in
$B(\cH)$. For self-adjoint $A,B\in B(\cH)$, $A\le B$ means that $B-A\in B(\cH)^+$. The
\emph{Thompson metric} $\dT$ on $\bP$ is defined by
$$
\dT(A,B):=\log\max\{M(A/B),M(B/A)\}=\|\log
A^{-1/2}BA^{-1/2}\|,\quad A,B\in\bP,
$$
where $\|\cdot\|$ denotes the operator norm on $B(\cH)$ and
$M(A/B):=\inf\{\alpha>0:A\le\alpha B\}$. The $\dT$-topology is equivalent to the operator norm
topology on $\bP$, and $(\bP,\dT)$ becomes a complete metric space, see \cite{Th}. On the other
hand, the strong operator topology is denoted by \emph{SOT}.

Let $\cP(\bP)$ be the set of Borel probability measures $\mu$ on $\bP$ with full support, i.e.,
$\mu(\supp(\mu))=1$, where $\supp(\mu)$ denotes the support of $\mu$. We denote by $\cPP$ the
set of $\mu\in\cP(\bP)$ whose support is bounded in the sense that the support of $\mu$ is included in
$$
\Sigma_\eps:=\{A\in\bP:\eps I\le A\le\eps^{-1}I\}
$$
for some $\eps\in(0,1)$. For any $\eps\in(0,1)$ the subset $\Sigma_\eps$ of $\bP$ with SOT
is metrizable by the metric
\begin{align}\label{F-2.1}
d_\eps(A,B):=\sum_{n=1}^\infty{1\over2^n}\|(A-B)x_n\|,\qquad A,B\in\Sigma_\eps,
\end{align}
where $\{x_n\}_{n=1}^\infty$ is dense in $\{x\in\cH:\|x\|\le1\}$,
see, e.g., \cite[p.\,262]{Co}. As easily verified, the above metric
is complete on $\Sigma_\eps$ so that $(\Sigma_\eps,d_\eps)$ becomes a
Polish space.

\begin{definition}\label{D-2.1}\rm
A set $\cU\subset\bP$ is said to be an \emph{upper set} if $B\in\bP$ and $B\ge A$ for some
$A\in\cU$ then $B\in\cU$. Also, a set $\cL\subset\bP$ is a \emph{lower set} if $B\in\bP$
and $B\le A$ for some $A\in\cL$ then $B\in\cL$. For $\mu,\nu\in\cP(\bP)$ we write $\mu\le\nu$
if $\mu(\cU)\le\nu(\cU)$ for every upper closed set $\cU$, or equivalently, if
$\mu(\cL)\ge\nu(\cL)$ for every lower closed set $\cL$. It is known
\cite[Propositions 3.6 and 3.11]{HLL} that $\mu\le\nu$ if and only if
$\int_\bP f(A)\,d\mu(A)\le\int_\bP f(A)\,d\nu(A)$ for any monotone (bounded) Borel function
$f:\bP\to\bR^+:=[0,\infty)$, or equivalently, for any monotone (bounded) continuous (in the
operator norm) function $f:\bP\to\bR^+$. Here, $f$ is \emph{monotone} if $A\le B$ in $\bP$
implies $f(A)\le f(B)$.
\end{definition}

\begin{lemma}\label{L-2.2}
Assume that $\mu_1,\mu_2\in\cPP$ and $\mu_1\le\mu_2$. Then there exists an $\eps\in(0,1)$
such that all $\mu\in\cPP$ with $\mu_1\le\mu\le\mu_2$ are supported on $\Sigma_\eps$.
\end{lemma}

\begin{proof}
Choose an $\eps>0$ such that $\mu_1,\mu_2$ are supported on $\Sigma_\eps$. Let
$\mu\in\cPP$ be such that $\mu_1\le\mu\le\mu_2$. Since $\{A\in\bP:A\ge\eps I\}$ is an upper
closed set, $\mu(A\ge\eps I)\ge\mu_1(A\ge\eps I)=1$. Since $\{A\in\bP:A\le\eps^{-1}I\}$ is
a lower closed set, $\mu(A\le\eps^{-1}I)\ge\mu_2(A\le\eps^{-1}I)=1$. Hence
$\mu(\eps I\le A\le\eps^{-1}I)=1$, so $\mu$ is supported on $\Sigma_\eps$.
\end{proof}

In this paper, the next notion of monotone convergence for a sequence of probability
measures in $\cPP$ will play an important role.

\begin{definition}\label{D-2.3}\rm
For $\mu,\mu_k\in\cPP$ ($k\in\bN$) we write
$$
\mu_k\nearrow\mu\qquad(\mbox{resp.,}\ \ \mu_k\searrow\mu)
$$
if the following conditions are satisfied:
\begin{itemize}
\item[(a)] $\mu_1\le\mu_2\le\cdots\le\mu$ (resp., $\mu_1\ge\mu_2\ge\dots\ge\mu$) in the sense
of Definition \ref{D-2.1}. In this case, since $\mu_1\le\mu_k\le\mu$ (resp.,
$\mu_1\ge\mu_k\ge\mu$) for all $k$, by Lemma \ref{L-2.2} there is an $\eps\in(0,1)$ such that
$\mu$ and $\mu_k$ are all supported on $\Sigma_\eps$.
\item[(b)] For any bounded SOT-continuous real function $f$ on $\Sigma_\eps$ where $\eps$ is
chosen in (a),
$$
\int_\bP f(A)\,d\mu_k(A)\ \longrightarrow\ \int_\bP f(A)\,d\mu(A)\quad
\mbox{as $k\to\infty$}.
$$
\end{itemize}
Note that condition (b) is independent of the choice of $\eps\in(0,1)$ in (a); in fact, when
$0<\eps'<\eps$, any bounded SOT-continuous real function on $\Sigma_\eps$ can extend to a
bounded SOT-continuous function on $\Sigma_{\eps'}$. The convergence $\mu_k\to\mu$ in (b)
reduces to the usual weak convergence as Borel probability measures on the Polish space
$(\Sigma_\eps.d_\eps)$ with $d_\eps$ in \eqref{F-2.1}.
\end{definition}

One can define a variant of monotone convergence for probability measures in $\cPP$ by
replacing the SOT-continuity for $f$ with operator norm continuity. When $\cH$ is
infinite-dimensional, the monotone convergence for probability measures in the SOT sense
is strictly weaker than that in the norm sense; in fact, for point measures
$\delta_A,\delta_{A_k}$ with $A,A_k\in\bP$, $\delta_{A_k}\nearrow\delta_A$ in Definition
\ref{D-2.3} means that $A_k\nearrow A$ in SOT, while $\delta_{A_k}\nearrow\delta_A$ in the
norm sense implies that $A_k\to A$ in the operator norm. The monotone convergence in the SOT
sense adopted in Definition \ref{D-2.3} is essential in this paper.

\begin{remark}\label{R-2.4}\rm
The assumption of the Hilbert space $\cH$ being separable is not essential in the paper.
Indeed, when $\cH$ is a general Hilbert space, it is known \cite[Lemma 2.1]{La} that any
probability measure on $(\bP,\cB(\bP))$ has the separable support. All of our results in the
paper are concerned with an at most countable set $\{\mu_k\}$ in $\cPP$. For such $\mu_k$'s
there exists a separable closed subset $\mathcal{X}$ of $\bP$ such that all $\mu_k$ are
supported on $\mathcal{X}$. The $C^*$-algebra generated by $\mathcal{X}$ is faithfully
represented on a separable Hilbert space $\cH_0$, so that we may regard $\mu_k$'s as
probability measures on $\bP(\cH_0)$. Thus we can reduce all our arguments to the separable
Hilbert space case.
\end{remark}

The notion of two-variable operator means was introduced by Kubo and Ando \cite{KA} in an
axiomatic way. A map $\sigma:B(\cH)^+\times B(\cH)^+\to B(\cH)^+$ is called an
\emph{operator mean} if it satisfies the following properties:
\begin{itemize}
\item[(I)] \emph{Monotonicity}: $A\le C$, $B\le D$ $\implies$ $A\sigma B\le C\sigma D$.
\item[(II)] \emph{Transformer inequality}: $C(A\sigma B)C\le(CAC)\sigma(CBC)$ for every
$C\in B(\cH)^+$.
\item[(III)] \emph{Downward continuity in SOT}: $A_k\searrow A$, $B_k\searrow B$ $\implies$
$A_k\sigma B_k\searrow A\sigma B$.
\item[(IV)] \emph{Normalized condition}: $I\sigma I=I$.
\end{itemize}

Each operator mean $\sigma$ is associated with a
positive operator monotone function $f$ on $(0,\infty)$ with
$f(1)=1$ in such a way that
$$
A\sigma B=A^{1/2}f(A^{-1/2}BA^{-1/2})A^{1/2},\qquad A,B\in\bP,
$$
which extends to general $A,B\in B(\cH)^+$ as $A\sigma
B=\lim_{\eps\searrow0}(A+\eps I)\sigma(B+\eps I)$ in SOT. Here, a
function $f$ on $(0,\infty)$ is \emph{operator monotone} if $A\le B$
$\implies$ $f(A)\le f(B)$ for $A,B\in\bP$. The above operator
monotone function $f$ on $(0,\infty)$ corresponding to $\sigma$ is
denoted by $f_\sigma$ and called the \emph{representing function} of
$\sigma$. Note that $f_\sigma$ is analytic on $(0,\infty)$ with
$f_\sigma'(1)\in[0,1]$ and $f_\sigma'(1)=0$ only when
$f_\sigma\equiv1$. A concise exposition on two-variable operator
means is found in \cite[Chapter 3]{Hi}.

In this paper we shall consider a certain extension of operator means of two variables to
those of probability measures in $\cPP$. The following is the definition of such operator means
of probability measures with minimally required properties, while we shall add further
properties accordingly when needed.

\begin{definition}\label{D-2.5}\rm
We say that a map
$$
M:\cPP\ \longrightarrow\ \bP
$$
is an \emph{operator mean on $\cPP$} if it satisfies the following properties:
\begin{itemize}
\item[(i)] \emph{Monotonicity}: If $\mu,\nu\in\cPP$ and $\mu\le\nu$ in the sense of
Definition \ref{D-2.1},
then $M(\mu)\le M(\nu)$.

\item[(ii)] \emph{Positive homogeneity}: $M(\alpha.\mu)=\alpha M(\mu)$ for every
$\mu\in\cPP$ and $\alpha>0$, where $\alpha.\mu$ is the push-forward
of $\mu$ by the map $A\in\bP\mapsto\alpha A$.

\item[(iii)] \emph{Monotone continuity}: For $\mu,\mu_k\in\cPP$ ($k\in\bN$), if either
$\mu_k\searrow\mu$ or $\mu_k\nearrow\mu$ (in the sense of Definition \ref{D-2.3}), then
$M(\mu_k)\to M(\mu)$ in SOT.

\item[(iv)] \emph{Normalized condition}: $M(\delta_I)=I$.
\end{itemize}
\end{definition}

Resemblances of properties (i)--(iv) to (I)--(IV) for Kubo-Ando's
two-variable operator means are apparent, but there are also slight
differences between those. For one thing, operator means in
Definition \ref{D-2.5} are maps on $\cPP$ while two-variable
operator means are on $B(\cH)^+\times B(\cH)^+$ permitting
non-invertible operators. For another, (ii) is weaker than (II).
Here we note \cite{KA} that congruence invariance $S^*(A\sigma
B)S=(S^*AS)\sigma(S^*BS)$ for invertible $S\in B(\cH)$ is automatic
when $\sigma$ is a two-variable operator mean. Moreover, we assume
continuity both downward and upward in (iii) while only downward is
assumed in (III). Continuity from both directions seems natural when
we take care of transformation under $A\in\bP\mapsto A^{-1}\in\bP$
for operator means on $\cPP$. Note also that $\sigma$ is upward
continuous when restricted to $\bP\times\bP$, while it is not
necessarily SOT-continuous on $\bP\times\bP$. Since these facts do
not seem widely known, we supply the details in Appendix A for the
convenience of the reader.

In the rest of the section, we will show that every operator mean on $\cPP$ is
contractive for the $\infty$-Wasserstein distance on $\cPP$. To do so, we first recall
some relevant notions in the setting of a general complete metric space $(X,d)$. Let
$\cP(X)$ be the set of Borel probability measures $\mu$ on $X$ with full support. For
$1\le p<\infty$ let $\cP^p(X)$ be the set of $\mu\in\cP(X)$ with finite $p$th moment, i.e.,
$\int_Xd^p(x,y)\,d\mu(x)<\infty$ for some (hence for all) $y\in X$. Moreover, let
$\cP^\infty(X)$ be the set of $\mu\in\cP(X)$ with bounded support, i.e.,
$\supp(\mu)\subset\{x\in X:d(x,y)\le R\}$ for some $y\in X$ and some $R>0$. Obviously,
$$
\cP^\infty(X)\subset\cP^q(X)\subset\cP^p(X)\subset\cP^1(X)\quad(1<p<q<\infty).
$$

For $1\le p<\infty$, the \emph{$p$-Wasserstein distance} $d_p^W$ on
the set $\cP^p(X)$ is defined as
\begin{align}\label{F-2.2}
d_p^W(\mu_1,\mu_2):=\biggl[\inf_{\pi\in\Pi(\mu_1,\mu_2)}
\int_{X\times X}d^p(x,y)\,d\pi(x,y)\biggr]^{1/p},\quad\mu_1,\mu_2\in\cP^p(\bP),
\end{align}
where $\Pi(\mu_1,\mu_2)$ is the set of all couplings for $\mu_1,\mu_2$ (i.e.,
Borel probability measures on $\bP\times\bP$ whose marginals are $\mu_1,\mu_2$).
It is well-known that $d_p^W$ is a complete metric on $\cP^p(X)$. See, e.g., \cite{St,Vi}
for more details on $d_p^W$. The $\infty$-version of $d_p^W$ is the
\emph{$\infty$-Wasserstein distance} $d_\infty^W$ on $\cP^\infty(X)$ defined as
$$
d_\infty^W(\mu_1,\mu_2):=\inf_{\pi\in\Pi(\mu_1,\mu_2)}
\sup\{d(x,y):(x,y)\in\supp(\pi)\},\quad\mu_1,\mu_2\in\cP^\infty(X).
$$
It is also known that $d_\infty^W$ is a complete metric on $\cP^\infty(X)$ and for every
$\mu,\nu\in\cP^\infty(X)$,
$$
d_\infty^W(\mu,\nu)=\lim_{p\to\infty}d_p^W(\mu,\nu)\quad\mbox{increasingly}.
$$
To prove these facts on $d_\infty^W$, one can assume that $(X,d)$ is a Polish space;
then the proof is found in \cite[Theorem 2.8]{BL}.

\begin{definition}\label{D-2.6}\rm
Let $1\le p\le\infty$. A map $\beta:\cP^p(X)\to X$  is called a \emph{barycentric map}
or a \emph{barycenter} on $\cP^p(X)$ if $\beta(\delta_x)=x$ for all $x\in X$. We say
that the map $\beta$ is \emph{$d_p^W$-contractive} if
$$
d(\beta(\mu),\beta(\nu))\leq d_p^W(\mu,\nu)
$$
for all $\mu,\nu\in\cP^p(X)$.
\end{definition}

Next, we consider a more specialized situation of an ordered metric
space with the Thompson metric. Let $E$ be a Banach space including
an open convex cone $C$ such that its closure $\overline C$ is a
proper cone, i.e., $\overline C\cap(-\overline C)=\{0\}$. The cone
$\overline C$ defines a closed partial order on $E$ by $x\leq y$ if
$y-x\in\overline C$.  The cone $\overline C$ is said to be
\emph{normal} if there is a constant $K$ such that $0\leq x\leq y$
implies $\Vert x\Vert\leq K\Vert y\Vert$. A typical case of $(E,C)$
is $(B(\cH),\bP(\cH))$, which is our setting of this paper.
The \emph{Thompson metric} $\dT$ on
$C$ is defined by
$$
\dT(x,y):=\log\max\{M(x/y),M(y/x)\}=\min\{r\ge0:e^{-r}y\le x\le
e^ry\},
$$
where $M(x/y):=\inf\{\alpha>0:x\le\alpha y\}$. As is well-known \cite{Th},
$\dT$ is a complete metric on $C$ and the $\dT$-topology agrees with
the relative norm topology on $C$. For $\mu,\nu\in\cP(C)$, the
\emph{stochastic order} $\mu\le\nu$ is defined as in Definition \ref{D-2.1}.
Note \cite[Theorem 4.3]{HLL} that $\mu\le\nu$ is a partial order on $\cP(C)$.
For every $\mu,\nu\in\cP^\infty(C)$, there is an $\alpha\in[1,\infty)$ such that
$\alpha^{-1}.\nu\le\mu\le\alpha.\mu$, where $\alpha.\nu$ is the push-forward of
$\nu$ by the map $x\in\ C\mapsto\alpha x\in C$. Hence one can define the
Thompson metric-like function as
$$
\delta_{\mathrm{T}}(\mu,\nu):=\inf\{r\ge0:e^{-r}.\nu\le\mu\le
e^r.\nu\}, \qquad\mu,\nu\in\cP^\infty(C).
$$

\begin{prop}\label{P-2.7}
$\delta_{\mathrm{T}}(\mu,\nu)$ is a metric on $\cP^\infty(C)$ and
for every $\mu,\nu\in\cP^\infty(C)$,
\begin{align}\label{F-2.3}
\delta_{\mathrm{T}}(\mu,\nu)\le d_\infty^W(\mu,\nu).
\end{align}
\end{prop}

\begin{proof}
Let $\mu,\nu,\lambda\in\cP^\infty(C)$. It is obvious that
$\delta_{\mathrm{T}}\ge0$ and
$\delta_{\mathrm{T}}(\mu,\nu)=\delta_{\mathrm{T}}(\nu,\mu)$. If
$\delta_{\mathrm{T}}(\mu,\nu)=0$, then $\nu\le\mu\le\nu$, which implies
$\mu=\nu$ by \cite[Theorem 4.3]{HLL}. To prove the triangle
inequality, let $r:=\delta_{\mathrm{T}}(\mu,\nu)$ and
$t:=\delta_{\mathrm{T}}(\nu,\lambda)$. Since $e^{-r}.\nu\le\mu\le
e^r.\nu$ and $e^{-t}.\lambda\le\nu\le e^t.\lambda$, we have $\mu\le
e^r.(e^t.\lambda)=e^{r+t}.\lambda$ and $\mu\ge
e^{-r}.(e^{-t}.\lambda)=e^{-(r+t)}.\lambda$, so
$\delta_{\mathrm{T}}(\mu,\lambda)\le r+t$.

Next, we prove inequality \eqref{F-2.3}. For any $\rho>d_\infty^W(\mu,\nu)$, one can
choose a $\pi\in\Pi(\mu,\nu)$ such that $\sup\{\dT(x,y):(x,y)\in\supp(\pi)\}<\rho$.
We prove that $e^{-\rho}.\nu\le\mu\le e^\rho.\nu$. For every upper closed set $\cU$
in $C$, note that
\begin{align*}
\mu(\cU)&=\pi((\cU\times C)\cap\supp(\pi)), \\
(e^\rho.\nu)(\cU)&=\nu(e^{-\rho}\cU)=\pi(C\times(e^{-\rho}\cU)).
\end{align*}
Assume that $(x,y)\in(\cU\times C)\cap\supp(\pi)$. Since $x\in\cU$
and $\dT(x,y)<\rho$ so that $x\le e^\rho y$, we have $e^\rho
y\in\cU$, implying $(x,y)\in C\times(e^{-\rho}\cU)$. Hence
$(\cU\times C)\cap\supp(\pi)\subset C\times(e^{-\rho}\cU)$, which
implies that $\mu(\cU)\le(e^\rho.\nu)(\cU)$. This means that $\mu\le
e^\rho.\nu$, and similarly $\nu\le e^\rho.\mu$ or
$e^{-\rho}.\nu\le\mu$. Therefore, we have
$\delta_{\mathrm{T}}(\mu,\nu)\le\rho$, giving (2.3).
\end{proof}

Assume that a map $\beta:\cP^\infty(C)\to C$ satisfies monotonicity
and positive homogeneity as in (i) and (ii) of Definition
\ref{D-2.5}. If $\mu,\nu\in\cP^\infty(C)$ and $e^{-r}.\nu\le\mu\le
e^r.\nu$ with $r\ge0$, then $e^{-r}\beta(\nu)\le\beta(\mu)\le
e^r\beta(\nu)$ so that $\dT(\beta(\mu),\beta(\nu))\le r$. Hence
$\dT(\beta(\mu),\beta(\nu))\le\delta_{\mathrm{T}}(\mu,\nu)$. From
this and Proposition \ref{P-2.7} we have

\begin{thm}\label{T-2.8}
An operator mean $M$ on $\cPP$ is
$(\dT)_\infty^W$-contractive$;$ in fact, for every
$\mu,\nu\in\cPP$,
$$
\dT(M(\mu),M(\nu))\le\delta_{\mathrm{T}}(\mu,\nu)\le
(\dT)_\infty^W(\mu,\nu).
$$
\end{thm}

\section{Deformed operator means}

Throughout the section, let $M$ be an operator mean on $\cPP$ as introduced in Definition
\ref{D-2.5}. For any two-variable operator mean $\sigma$ (in the Kubo-Ando sense) and any
$\mu\in\cPP$, we consider the fixed point type equation
\begin{align}\label{F-3.1}
X=M(X\sigma\mu),\qquad X\in\bP,
\end{align}
where $X\sigma\mu$ is the push-forward of $\mu$ by the map $A\in\bP\mapsto X\sigma A\in\bP$.
Note that it is easily seen that if $\mu\in\cPP$ and $X\in\bP$, then $X\sigma\mu\in\cPP$, so
the above equation makes sense. In the rest of this section we shall prove the following:

\begin{thm}\label{T-3.1}
Assume that $\sigma\ne\frak{l}$, where $\frak{l}$ is the left trivial two-variable operator
mean, i.e., $X\frak{l}Y=X$ for all $X,Y\in\bP$.
\begin{itemize}
\item[(1)] For every $\mu\in\cPP$ there exists a unique $X_0\in\bP$ satisfying \eqref{F-3.1}.
\item[(2)] If $Y\in\bP$ satisfies $Y\ge M(Y\sigma\mu)$, then $Y\ge X_0$, and if $Y'\in\bP$
satisfies $Y'\le M(Y'\sigma\mu)$, then $Y'\le X_0$.
\item[(3)] Write $M_\sigma(\mu)$ for the solution $X_0$. Then the map $M_\sigma:\cPP\to\bP$
satisfies {\rm(i)--(iv)}, that is, $M_\sigma$ is an operator mean on $\cPP$ again.
\end{itemize}
\end{thm}

We call $M_\sigma:\cPP\to\bP$ given in the theorem the \emph{deformed operator mean} from $M$
by $\sigma$.

\begin{remark}\label{R-3.2}\rm
It is easy to verify that the arithmetic mean $\cA(\mu):=\int_\bP A\,d\mu(A)$ on
$\cPP$ satisfies (i)--(iv). When $M=\cA$, equation \eqref{F-3.1} is equivalently written as
\begin{align}\label{F-3.2}
X=\int_\bP X\sigma A\,d\mu(A),\quad\mbox{i.e.},\quad I=\int_\bP
f_\sigma(X^{-1/2}AX^{-1/2})\,d\mu(A),
\end{align}
where $f_\sigma$ is the representing function of $\sigma$ so that
$X\sigma A=X^{1/2}f_\sigma(X^{-1/2}AX^{-1/2})X^{1/2}$ for $X,A\in\bP$. Here we remark that
the functions $A\in\bP\mapsto X\sigma A,f_\sigma(X^{-1/2}AX^{-1/2})\in\bP$ are continuous
in the operator norm and are operator norm bounded on the support of $\mu$. Hence the above
integrals $\int_\bP A\,d\mu(A)$, $\int_\bP X\sigma A\,d\mu(A)$ and
$\int_\bP f_\sigma(X^{-1/2}AX^{-1/2})\,d\mu(A)$ are well defined as Bochner integrals.

Assume that $\sigma\ne\frak{l}$ so that $f_\sigma'(1)>0$, and set
$g_\sigma(x):=(f_\sigma(x)-1)/f_\sigma'(1)$. Then $g_\sigma$ is an
operator monotone function with $g_\sigma(1)=0$ and
$g_\sigma'(1)=1$, and \eqref{F-3.1} or \eqref{F-3.2} is equivalent
to
$$
\int_\bP g_\sigma(X^{-1/2}AX^{-1/2})\,d\mu(A)=0.
$$
Hence, equation \eqref{F-3.1} in this case reduces to the \emph{generalized Karcher equation}
in \cite[Definition 2.2]{Pa}. The method in \cite{Pa} is relied on the Banach contraction
principle, while our proof of the theorem will be done by a simple argument of monotone
convergence. Hence our proof is applicable to any operator mean $M$ satisfying (i)--(iv).
\end{remark}

\begin{remark}\label{R-3.3}\rm
As discussed in \cite{Hi2,HSW} the deformed operator means can be also considered in the
restricted setting of two-variable operator means (in the Kubo-Ando sense) and in that of
$n$-variable operator means. As for the two-variable case, when $M=\tau$ is a two-variable
operator mean, the reduced equation of \eqref{F-3.1} is, for $\sigma\ne\frak{l}$ and
$A,B\in\bP$,
$$
X=(X\sigma A)\tau(X\sigma B),\qquad X\in\bP.
$$
which has a unique solution $X_0\in\bP$ as in Theorem \ref{T-3.1}. If we write
$A\tau_\sigma B$ for the solution $X_0$, then $\tau_\sigma$ becomes a two-variable operator
mean again and the representing function of $\tau_\sigma$ is exactly determined by those of
$\tau$ and $\sigma$, see \cite{Hi2}. The restriction to two-variable operator means was
also discussed in \cite{Pa,Ya2} for the generalized Karcher equation mentioned in Remark
\ref{R-3.2}.
\end{remark}

The following proof of the theorem is essentially the same as that
of \cite[Theorem 2.1]{HSW} (also \cite{Hi2}), where a similar
theorem was shown for multivariate operator means $M:\bP^n\to\bP$.
Now, let $M$ and $\sigma$ be as in the theorem. We first give some
lemmas.

\begin{lemma}\label{L-3.4}
Let $\ffi,\psi:\bP\to\bP$ be monotone and Borel measurable. Assume that $\ffi(A)\le\psi(A)$
for all $A\in\bP$. If $\mu,\nu\in\cPP$ and $\mu\le\nu$, then $\ffi_*\mu,\psi_*\nu\in\cPP$
and $\ffi_*\mu\le\psi_*\nu$.
\end{lemma}

\begin{proof}
Let $\ffi,\psi$ and $\mu,\nu$ be as stated. That
$\ffi_*\mu,\psi_*\nu\in\cPP$ follows immediately from $\ffi,\psi$
being monotone. For any monotone bounded Borel function
$f:\bP\to\bR^+$, $f(\ffi(A))\le f(\psi(A))$ for all $A\in\bP$ and
$f\circ\psi$ is monotone. Hence it follows from \cite[Proposition 3.6]{HLL} that
\begin{align*}
\int_\bP f(A)\,d(\ffi_*\mu)(A)&=\int_\bP f(\ffi(A))\,d\mu(A)
\le\int_\bP f(\psi(A))\,d\mu(A) \\
&\le\int_\bP f(\psi(A))\,d\nu(A)=\int_\bP f(A)\,d(\psi_*\nu)(A),
\end{align*}
which gives $\ffi_*\mu\le\psi_*\nu$ by \cite[Proposition 3.6]{HLL} again.
\end{proof}

\begin{lemma}\label{L-3.5}
Let $\sigma$ be arbitrary.
\begin{itemize}
\item[(1)] Let $X,X_k\in\bP$ ($k\in\bN$), and assume that $X_1\ge X_2\ge\cdots$ $($resp.,
$X_1\le X_2\le\cdots)$ and $X_k\to X$ in SOT. Then for every
$\mu\in\cPP$, $X_k\sigma\mu\searrow X\sigma\mu$ $($resp.,
$X_k\sigma\mu\nearrow X\sigma\mu)$ in the sense of Definition
$\ref{D-2.3}$.
\item[(2)] Let $\mu,\mu_k\in\cPP$ ($k\in\bN$), and assume that $\mu_k\searrow\mu$ $($resp.,
$\mu_k\nearrow\mu)$. Then for every $X\in\bP$, $X\sigma\mu_k\searrow
X\sigma\mu$ $($resp., $X\sigma\mu_k\nearrow X\sigma\mu)$.
\end{itemize}
\end{lemma}

\begin{proof}
(1)\enspace From Lemma \ref{L-3.4} it is immediate that
$X_1\sigma\mu\ge X_2\sigma\mu\ge\cdots\ge X\sigma\mu$ (resp.,
$X_1\sigma\mu\le X_2\sigma\mu\le\cdots\le X\sigma\mu$). Choose an
$\eps\in(0,1)$ such that all of $X,X_k$ and $\supp(\mu)$ are in
$\Sigma_\eps$. Since $Y\sigma A\in\Sigma_\eps$ for all
$Y,A\in\Sigma_\eps$, note that all $X\sigma\mu$ and $X_k\sigma\mu$
are supported on $\Sigma_\eps$. Since $X_k\sigma
A\to X\sigma A$ in SOT as $k\to\infty$ for any $A\in\bP,$ for every
bounded SOT-continuous real function $f$ on $\Sigma_\eps$, it
follows from the bounded convergence theorem that
\begin{align*}
&\int_\bP f(A)\,d(X_k\sigma\mu)(A)=\int_\bP f(X_k\sigma A)\,d\mu(A) \\
&\qquad\longrightarrow\ \int_\bP f(X\sigma A)\,d\mu(A)=\int_\bP f(A)\,d(X\sigma\mu)(A).
\end{align*}
This implies the assertion.

(2)\enspace From Lemma \ref{L-3.4} again, $X\sigma\mu_1\ge
X\sigma\mu_2\ge\cdots\ge X\sigma\mu$ (resp., $X\sigma\mu_1\le
X\sigma\mu_2\le\cdots\le X\sigma\mu$). Choose an $\eps\in(0,1)$ such
that all $\mu,\mu_k$ are supported on $\Sigma_\eps$ and
$X\in\Sigma_\eps$. Then all $X\sigma\mu$ and $X\sigma\mu_k$ are
supported on $\Sigma_\eps$. For every bounded SOT-continuous real
function $f$ on $\Sigma_\eps$, since $A\in\Sigma_\eps\mapsto
f(X\sigma A)$ is SOT-continuous, we have
\begin{align*}
&\int_\bP f(A)\,d(X\sigma\mu_k)(A)=\int_\bP f(X\sigma A)\,d\mu_k(A) \\
&\qquad\longrightarrow\ \int_\bP f(X\sigma A)\,d\mu(A)=\int_\bP f(A)\,d(X\sigma\mu)(A).
\end{align*}
Hence the assertion follows.
\end{proof}

The next lemma is crucial to obtain the uniqueness of a solution to \eqref{F-3.1}.

\begin{lemma}\label{L-3.6}
Let $\sigma\ne\frak{l}$. If $X,Y\in\bP$ and $X\ne Y$, then
$$
\dT(M(X\sigma\mu),M(Y\sigma\mu))<\dT(X,Y)
$$
for every $\mu\in\cPP$.
\end{lemma}

\begin{proof}
Let $X,Y\in\bP$ be such that $X\ne Y$, and let $\alpha:=\dT(X,Y)>0$. Choose an
$\eps\in(0,1)$ such that $X$ and $\supp(\mu)$ are in $\Sigma_\eps$. For every $A\in\bP$
note that
\begin{align}
Y\sigma A&\le(e^\alpha X)\sigma A
=e^\alpha X^{1/2}f_\sigma(e^{-\alpha}X^{-1/2}AX^{-1/2})X^{1/2}, \label{F-3.3}\\
Y\sigma A&\ge(e^{-\alpha}X)\sigma A
=e^{-\alpha}X^{1/2}f_\sigma(e^\alpha X^{-1/2}AX^{-1/2})X^{1/2}. \label{F-3.4}
\end{align}
For every $A\in\Sigma_\eps$, since $\eps^2I\le X^{-1/2}AX^{-1/2}\le\eps^{-2}I$, we have
\begin{align*}
f_\sigma(X^{-1/2}AX^{-1/2})-f_\sigma(e^{-\alpha}X^{-1/2}AX^{-1/2})
&\ge\biggl(\min_{t\in[\eps^2,\eps^{-2}]}\{f_\sigma(t)-f_\sigma(e^{-\alpha}t)\}\biggr)I, \\
f_\sigma(e^\alpha X^{-1/2}AX^{-1/2})-f_\sigma(X^{-1/2}AX^{-1/2})
&\ge\biggl(\min_{t\in[\eps^2,\eps^{-2}]}\{f_\sigma(e^\alpha t)-f_\sigma(t)\}\biggr)I.
\end{align*}
Since $\sigma\ne\frak{l}$, $f_\sigma$ is strictly increasing (and analytic) on $(0,\infty)$,
the minima in the above two expressions are strictly positive. Hence there exists a
$\rho\in(0,1)$ such that, for every $A\in\Sigma_\eps$,
\begin{align}
f_\sigma(X^{-1/2}AX^{-1/2})-f_\sigma(e^{-\alpha}X^{-1/2}AX^{-1/2})&\ge\rho I, \label{F-3.5}\\
f_\sigma(e^\alpha X^{-1/2}AX^{-1/2})-f_\sigma(X^{-1/2}AX^{-1/2})&\ge\rho I. \label{F-3.6}
\end{align}
Therefore,
\begin{align*}
Y\sigma A&\le e^\alpha(X\sigma A-\rho X)\le e^\alpha(1-\rho\eps^2)(X\sigma A), \\
Y\sigma A&\ge e^{-\alpha}(X\sigma A+\rho X)\ge e^{-\alpha}(1+\rho\eps^2)(X\sigma A),
\end{align*}
since $\eps I\le X\le\eps^{-1}I$ and $\eps I\le X\sigma A\le\eps^{-1}I$ so that
$\eps^2(X\sigma A)\le X\le\eps^{-2}(X\sigma A)$. Choosing a $\beta\in(0,\alpha)$ such that
$e^{\alpha-\beta}\le1+\rho\eps^2$, we have
\begin{align}\label{F-3.7}
e^{-\beta}(X\sigma A)\le Y\sigma A\le e^\beta(X\sigma A),\qquad A\in\Sigma_\eps.
\end{align}

Now, write $\psi_X(A):=X\sigma A$ and $\psi_Y(A):=Y\sigma A$ for $A\in\bP$. The claim given
in \eqref{F-3.7} means that $e^{-\beta}\psi_X(A)\le\psi_Y(A)\le e^\beta\psi_X(A)$ for all
$A\in\Sigma_\eps$. Since $(e^{\pm\beta}\psi_X)_*\mu=e^{\pm\beta}.(X\sigma\mu)$ and
$(\psi_Y)_*\mu=Y\sigma\mu$, Lemma \ref{L-3.4} gives
$$
e^{-\beta}.(X\sigma\mu)\le Y\sigma\mu\le e^\beta.(X\sigma\mu).
$$
By (i) and (ii) this implies that
$$
e^{-\beta}M(X\sigma\mu)\le M(Y\sigma\mu)\le e^\beta M(X\sigma\mu)
$$
so that $\dT(M(X\sigma\mu),M(Y\sigma\mu))\le\beta<\dT(X,Y)$.
\end{proof}

We are now in a position to prove the theorem.

\bigskip\noindent
\emph{Proof of Theorem $\ref{T-3.1}$.}\enspace (1)\enspace For any
fixed $\mu\in\cPP$ define a map $F:\bP\to\bP$ by
$$
F(X):=M(X\sigma\mu),\qquad X\in\bP,
$$
which is monotone, i.e., $X\le Y$ implies $F(X)\le F(Y)$, by Lemma \ref{L-3.4} and (i)
of Definition \ref{D-2.5}. Choose an $\eps\in(0,1)$ such that $\mu$ is supported on
$\Sigma_\eps$, and let $Z:=\eps^{-1}I$. Since $\mu\le\eps^{-1}.\delta_I$ and so
$Z\sigma\mu\le\eps^{-1}.\delta_I$, we have $F(Z)\le\eps^{-1}I=Z$ by (i) and (iv), and
iterating this implies that $Z\ge F(Z)\ge F^2(Z)\ge\cdots$. Moreover, since
$(\eps I)\sigma\mu\ge\eps.\delta_I$, $F(Z)\ge F(\eps I)\ge\eps I$, and by iterating this
we have $F^k(Z)\ge\eps I$ for all $k$. Therefore, $F^k(Z)\searrow X_0\in\bP$ for some
$X_0\in\bP$, and hence $F^k(Z)\sigma\mu\searrow X_0\sigma\mu$ by Lemma \ref{L-3.5}\,(1).
From the monotone continuity of $M$ in (iii) it follows that
$$
F^{k+1}(Z)=M(F^k(Z)\sigma\mu)\ \searrow\ M(X_0\sigma\mu),
$$
which yields that $X_0=M(X_0\sigma\mu)$.

To prove the uniqueness of the solution, assume that $X_1\in\bP$ satisfies
$X_1=M(X_1\sigma\mu)$ and $X_1\ne X_0$. Then by Lemma \ref{L-3.6} we have
$$
\dT(X_0,X_1)=\dT(M(X_0\sigma\mu),M(X_1\sigma\mu))<\dT(X_0,X_1),
$$
a contradiction.

(2)\enspace
Let $F$ be as in the proof of (1). If $Y\ge M(Y\sigma\mu)$, then
$Y\ge F(Y)\ge F^2(Y)\ge\cdots$. Choose an $\eps\in(0,1)$ such that $Y$ and $\supp(\mu)$
are in $\Sigma_\eps$. Then as in the proof of (1), $F^k(Y)\ge\eps I$ for all $k$, and
hence $F^k(Y)\searrow X'$ for some $X'\in\bP$. As in the proof of (1) again, $X'$ is a
solution to \eqref{F-3.1} so that $Y\ge X'=X_0$. The proof of the second assertion is
similar, where we have $Y'\le F(Y')\le F^2(Y')\le\cdots$ and use the upward continuity in
Lemma \ref{L-3.5}\,(1) and (iii).

(3)\enspace
(i)\enspace
Let $\mu,\nu\in\cPP$ and $\mu\le\nu$. Let $X_0:=M_\sigma(\mu)$ and $Y_0:=M_\sigma(\nu)$,
i.e., $X_0=M(X_0\sigma\mu)$ and $Y_0=M(Y_0\sigma\mu)$. Since $X_0\le M(X_0\sigma\nu)$ by
Lemma \ref{L-3.4}, we have $X_0\le Y_0$ by (2).

(ii)\enspace
One can easily see that $\alpha.(X\sigma\mu)=(\alpha X)\sigma(\alpha.\mu)$ for every
$X\in\bP$, $\mu\in\cPP$ and $\alpha>0$. Hence, if $X_0:=M_\sigma(\mu)$, then it follows
from (ii) for $M$ that
$$
\alpha X_0=M(\alpha.(X_0\sigma\mu))=M((\alpha X_0)\sigma(\alpha.\mu)),
$$
which gives $\alpha X_0=M_\sigma(\alpha.\mu)$.

(iii)\enspace Assume that $\mu,\mu_k\in\cPP$ and $\mu_k\searrow\mu$.
Let $X_k:=M_\sigma(\mu_k)$. Then $X_1\ge X_2\ge\cdots$ by (i) for
$M_\sigma$ already proved. By Lemma \ref{L-2.2}, there is an $\eps\in(0,1)$ such
that all $\mu_k$ are supported on $\Sigma_\eps$. From the
proof of (1) we have $X_k\ge\eps I$ for all $k$. Hence $X_k\searrow
X_0$ for some $X_0\in\bP$. It remains to prove that
$X_0=M_\sigma(\mu)$. For every $k$, since $X_k\sigma\mu_k\ge
X_0\sigma\mu$ by Lemma \ref{L-3.4}, we have
$X_k=M(X_k\sigma\mu_k)\ge M(X_0\sigma\mu)$. Hence $X_0\ge
M(X_0\sigma\mu)$ follows. On the other hand, when $k<l$, since
$X_l\sigma\mu_l\le X_l\sigma\mu_k$ by Lemma \ref{L-3.4} again, we
have
\begin{align}\label{F-3.8}
X_l=M(X_l\sigma\mu_l)\le M(X_l\sigma\mu_k).
\end{align}
For any fixed $k$, since $X_l\sigma\mu_k\searrow X_0\sigma\mu_k$ as $k<l\to\infty$ by
Lemma \ref{L-3.5}\,(1), it follows from (iii) for $M$ that
$M(X_l\sigma\mu_k)\to M(X_0\sigma\mu_k)$ in SOT. Hence from \eqref{F-3.8} we have
$X_0\le M(X_0\sigma\mu_k)$ for every $k$. Since $X_0\sigma\mu_k\searrow X_0\sigma\mu$ by
Lemma \ref{L-3.5}\,(2), we furthermore have $M(X_0\sigma\mu_k)\to M(X_0\sigma\mu)$ in SOT,
so that $X_0\le M(X_0\sigma\mu)$ follows. Therefore, we have shown that $X_0=M(X_0\sigma\mu)$,
that is, $X_0=M_\sigma(\mu)$.

When $\mu_k\nearrow\mu$, the proof is analogous, so we may omit the details.

(iv) is obvious since $I=M(\delta_I)=M(I\sigma\delta_I)$.\qed

\section{Basic properties}

In this section, as in Section 3, let $M$ be an operator mean on $\cPP$ and $\sigma$ be a
two-variable operator mean with $\sigma\ne\frak{l}$. The next properties of the deformed
operator mean $M_\sigma$ can easily been verified by using Theorem \ref{T-3.1}, whose proofs
are left to the reader.

\begin{prop}\label{P-4.1}
\begin{itemize}
\item[(1)] $M_{\frak{r}}=M$, where $\frak{r}$ is the right trivial two-variable operator mean
$X\frak{r}Y=Y$.
\item[(2)] Let $\hat M:\cPP\to\bP$ be an operator mean satisfying {\rm(i)--(iv)} and
$\hat\sigma$ be a two-variable operator mean with $\hat\sigma\ne\frak{l}$. If $M\le\hat M$
and $\sigma\le\hat\sigma$, then $M_\sigma\le\hat M_{\hat\sigma}$.
\item[(3)] Define the adjoint $M^*$ of $M$ by $M^*(\mu):=M(\mu^{-1})^{-1}$ for $\mu\in\cPP$,
where $\mu^{-1}$ is the push-forward of $\mu$ by $A\mapsto A^{-1}$, $A\in\bP$. Let
$\sigma^*$ be the adjoint of $\sigma$, i.e., $A\sigma^*B=(A^{-1}\sigma B^{-1})^{-1}$. Then
$M^*$ satisfies {\rm(i)--(iv)} again and $(M_\sigma)^*=(M^*)_{\sigma^*}$.
\end{itemize}
\end{prop}

In addition to properties (i)--(iv) of $M$ that are essential for Theorem \ref{T-3.1}, we
consider the following properties:
\begin{itemize}
\item[(v)] \emph{Barycentric identity}: $M(\delta_A)=A$ for every $A\in\bP$. This contains
(iv).

\item[(vi)] \emph{Congruence invariance}: For every $\mu\in\cPP$ and every invertible
$S\in B(\cH)$,
$$
SM(\mu)S^*=M(S\mu S^*),
$$
where $S\mu S^*$ is the push-forward of $\mu$ by $A\in\bP\mapsto SAS^*\in\bP$. This
property contains (ii).

\item[(vii)] \emph{Concavity}: For every $\mu_j,\nu_j\in\cPP$ ($1\le j\le n$), any weight
$(w_1,\dots,w_n)$ with any $n\in\bN$, and $0<t<1$,
$$
M\Biggl(\sum_{j=1}^nw_j(\mu_j\triangledown_t\nu_j)\Biggr)
\ge(1-t)M\Biggl(\sum_{j=1}^nw_j\mu_j\Biggr)+tM\Biggl(\sum_{j=1}^nw_j\nu_j\Biggr),
$$
where $\mu_j\triangledown_t\nu_j$ is the push-forward of $\mu_j\times\nu_j$ by
$\triangledown_t:\bP\times\bP\to\bP$, $\triangledown_t(A,B):=(1-t)A+tB$, the $t$-weighted
arithmetic mean. The following two particular cases may be worth noting separately. The first
one is the joint concavity when restricted to the weighted $n$-variable situation.
\begin{itemize}
\item[(vii-1)] For every $A_j,B_j\in\bP$ ($1\le j\le n$), and $0<t<1$,
$$
M\Biggl(\sum_{j=1}^nw_j\delta_{(1-t)A_j+tB_j}\Biggr)
\ge(1-t)M\Biggl(\sum_{j=1}^nw_j\delta_{A_j}\Biggr)
+tM\Biggl(\sum_{j=1}^nw_j\delta_{B_j}\Biggr).
$$
\item[(vii-2)] For every $\mu,\nu\in\cPP$ and $0<t<1$,
$$
M(\mu\triangledown_t\nu)\ge(1-t)M(\mu)+tM(\nu).
$$
\end{itemize}

\item[(viii)] \emph{Arithmetic-$M$-harmonic mean inequality}: For every $\mu\in\cPP$,
$$
\cH(\mu)\le M(\mu)\le\cA(\mu),
$$
where
\begin{align}\label{F-4.1}
\cA(\mu):=\int_\bP A\,d\mu(A),\qquad
\cH(\mu):=\biggl[\int_\bP A^{-1}\,d\mu(A)\biggr]^{-1}
\end{align}
are the arithmetic and the harmonic means.
\end{itemize}

\begin{thm}\label{T-4.2}
If $M$ satisfies each of {\rm(v)}, {\rm(vi)}, {\rm(vii)},
{\rm(vii-1)}, {\rm(vii-2)}, and {\rm(viii)} $($in addition to
{\rm(i)--(iv)}$)$, then $M_\sigma$ does the same.
\end{thm}

\begin{proof}
(v)\enspace
This is obvious since $A\sigma\delta_A=\delta_A$.

(vi)\enspace
Assume that $M$ satisfies (vi). For every $\mu\in\cPP$ let $X_0:=M_\sigma(\mu)$. Then for
any invertible $S\in B(\cH)$,
$$
SX_0S^*=SM(X_0\sigma\mu)S^*=M(S(X_0\sigma\mu)S^*)=M((SX_0S^*)\sigma(S\mu S^*))
$$
so that $SX_0S^*=M_\sigma(S\mu S^*)$.

(vii)\enspace
Assume that $M$ satisfies (vii). Let $\mu_j,\nu_j\in\cPP$ ($1\le j\le n$) and $(w_1,\dots,w_n)$
be any weight, and let $0<t<1$. Set $\mu:=\sum_{j=1}^nw_j\mu_j$, $\nu:=\sum_{j=1}^nw_j\nu_j$,
$X_0:=M_\sigma(\mu)$ and $Y_0:=M_\sigma(\nu)$. Since
$X_0\sigma\mu=\sum_{j=1}^nw_j(X_0\sigma\mu_j)$ and
$Y_0\sigma\nu=\sum_{j=1}^nw_j(Y_0\sigma\nu_j)$, we have
\begin{align}\label{F-4.2}
X_0\triangledown_tY_0=M(X_0\sigma\mu)\triangledown_tM(Y_0\sigma\nu)
\le M\Biggl(\sum_{j=1}^nw_j((X_0\sigma\mu_j)\triangledown_t(Y_0\sigma\nu_j))\Biggr)
\end{align}
thanks to (vii) for $M$. We now show that, for every $\mu,\nu\in\cPP$,
\begin{align}\label{F-4.3}
(X_0\sigma\mu)\triangledown_t(Y_0\sigma\nu)
\le(X_0\triangledown_tY_0)\sigma(\mu\triangledown_t\nu).
\end{align}
For $X\in\bP$ define $\psi_X:\bP\to\bP$ by $\psi_X(A):=X\sigma A$. The left-hand side of
\eqref{F-4.3} is
$$
(\triangledown_t)_*((\psi_{X_0})_*\mu\times(\psi_{Y_0})_*\nu)
=(\triangledown_t\circ(\psi_{X_0}\times\psi_{Y_0}))_*(\mu\times\nu).
$$
The right-hand side of \eqref{F-4.3} is
$$
(\psi_{X_0\triangledown_t Y_0}\circ\triangledown_t)_*(\mu\times\nu).
$$
Set $\ffi_1:=\triangledown_t\circ(\psi_{X_0}\times\psi_{Y_0})$ and
$\ffi_2:=\psi_{X_0\triangledown_t Y_0}\circ\triangledown_t$. To prove \eqref{F-4.3}, it
suffices to show that $\ffi_1(A,B)\le\ffi_2(A,B)$ for all $A,B\in\bP$. In fact, when this
holds, we have $(\ffi_1)_*(\mu\times\nu)\le(\ffi_2)_*(\mu\times\nu)$ similarly to the proof
of Lemma \ref{L-3.4}. For every $A,B\in\bP$ we have
\begin{align*}
\ffi_1(A,B)&=\psi_{X_0}(A)\triangledown_t\psi_{Y_0}(B)
=(X_0\sigma A)\triangledown_t(Y_0\sigma B) \\
&\le(X_0\triangledown_tY_0)\sigma(A\triangledown_tB)=\ffi_2(A,B),
\end{align*}
where inequality follows from \cite[Theorem 3.5]{KA}. Hence \eqref{F-4.3} has been shown,
from which we have
\begin{align*}
\sum_{j=1}^nw_j((X_0\sigma\mu_j)\triangledown_t(Y_0\sigma\nu_j))
&\le\sum_{j=1}^nw_j((X_0\triangledown_tY_0)\sigma(\mu_j\triangledown_t\nu_j)) \\
&=(X_0\triangledown_tY_0)\sigma\Biggl(\sum_{j=1}^nw_j(\mu_j\triangledown_t\nu_j)\Biggr).
\end{align*}
Applying monotonicity of $M$ to this and combining with \eqref{F-4.2} we have
$$
X_0\triangledown_tY_0\le M\Biggl((X_0\triangledown_tY_0)\sigma
\Biggl(\sum_{j=1}^nw_j(\mu_j\triangledown_t\nu_j)\Biggr)\Biggr),
$$
which implies that
$X_0\triangledown_tY_0\le M_\sigma\bigl(\sum_{j=1}^nw_j(\mu_j\triangledown_t\nu_j)\bigr)$ by
Theorem \ref{T-3.1}\,(2).

The proofs of the assertions for (vii-1) and for (vii-2) are similar
to the above proof for (vii), so we omit the details.

(viii)\enspace Assume that $M$ satisfies (viii). Let
$\alpha:=f_\sigma'(1)$; then $0<\alpha\le1$ since
$\sigma\ne\frak{l}$. It is well-known \cite[(3.3.2)]{Hi} that
$!_\alpha\le\sigma\le\triangledown_\alpha$, where $!_\alpha$ is the
$\alpha$-weighted harmonic mean $A!_\alpha B:=((1-\alpha)A^{-1}+\alpha B)^{-1}$.
By Proposition \ref{P-4.1}\,(2) we have
$$
\cH_{!_\alpha}\le M_\sigma\le\cA_{\triangledown_\alpha}.
$$
Hence it suffices to show that $\cH_{!_\alpha}=\cH$ and $\cA_{\triangledown_\alpha}=\cA$.
But it is immediate to find that the solutions of the equations
$X=\cA(X\triangledown_\alpha\mu)$ and $X=\cH(X\,!_\alpha\mu)$ are $\cA(\mu)$ and
$\cH(\mu)$, respectively.
\end{proof}

By Theorems \ref{T-2.8}, \ref{T-3.1} and \ref{T-4.2} we have

\begin{cor}\label{C-4.3}
If $M$ is a barycenter $($i.e., it satisfies {\rm(v)}$)$, then
$M_\sigma$ is a $(\dT)_\infty^W$-contractive
barycenter on $\cPP$ for any $\sigma\ne\frak{l}$.
\end{cor}

\section{Examples}

In this section we provide typical examples of operator means on
$\cPP$ satisfying (i)--(viii) and their deformed operator means. We
note from Corollary \ref{C-4.3} that all of those operator means are
$(\dT)_\infty^W$-contractive barycenters on $\cPP$.

\subsection{Arithmetic and harmonic means}

The arithmetic mean $\cA$ and the harmonic mean $\cH$ on $\cPP$ are given in \eqref{F-4.1}.
It is straightforward to see that $\cA$ satisfies all the properties (i)--(viii). It is also
easy to see $\cH$ satisfies the properties (i)--(viii) except (vii) (including (vii-1)
and (vii-2)). Since it does not seem easy to show (vii) directly for $\cH$, we take a detour
by giving the following proposition, which may be of independent interest.

\begin{prop}\label{P-5.1}
For every $\mu\in\cPP$, $\cA_{\,!_{s'}}(\mu)\le\cA_{\,!_s}(\mu)$ for $0<s'<s\le1$ and
$$
\cH(\mu)=\lim_{s\searrow0}\cA_{\,!_s}(\mu)\quad\mbox{in SOT}.
$$
\end{prop}

\begin{proof}
Let $\mu\in\cPP$ and set $X_s:=\cA_{\,!_s}(\mu)$ for each $s\in(0,1]$. Assume that
$0<s'<s\le1$. Since $X_s=\cA(X_s!_s\mu)$, we have
\begin{align}
I&=\cA(I!_s(X_s^{-1/2}\mu X_s^{-1/2})) \nonumber\\
&=\int_\bP\bigl[(1-s)I+s(X_s^{-1/2}AX_s^{-1/2})^{-1}\bigr]^{-1}\,d\mu(A) \nonumber\\
&=\int_\bP\bigl[I+s(X_s^{1/2}A^{-1}X_s^{1/2}-I)\bigr]^{-1}\,d\mu(A). \label{F-5.1}
\end{align}
Here note that $s\in(0,1]\mapsto(1+s(t-1))^{1/s}$ is a decreasing function for any $t>0$.
Hence we have
$$
\bigl[I+s(X_s^{1/2}A^{-1}X_s^{1/2}-I)\bigr]^{1/s}
\le\bigl[I+s'(X_s^{1/2}A^{-1}X_s^{1/2}-I)\bigr]^{1/s'}
$$
so that
\begin{align}\label{F-5.2}
\bigl[I+s(X_s^{1/2}A^{-1}X_s^{1/2}-I)\bigr]^{-s'/s}
\ge\bigl[I+s'(X_s^{1/2}A^{-1}X_s^{1/2}-I)\bigr]^{-1}.
\end{align}
Applying the operator concavity of $x^{s'/s}$ on $(0,\infty)$ to \eqref{F-5.1} and using
\eqref{F-5.2} we have
\begin{align*}
I&\ge\int_\bP\bigl[I+s(X_s^{1/2}A^{-1}X_s^{1/2}-I)\bigr]^{-s'/s}\,d\mu(A) \\
&\ge\int_\bP\bigl[I+s'(X_s^{1/2}A^{-1}X_s^{1/2}-I)\bigr]^{-1}\,d\mu(A) \\
&=\cA(I!_{s'}(X_s^{-1/2}\mu X_s^{-1/2}))
\end{align*}
so that $X_s\ge\cA(X_s!_{s'}\mu)$. Hence $X_s\ge X_{s'}$ by Theorem \ref{T-3.1}\,(2).

Choose an $\eps\in(0,1)$ such that $\mu$ is supported on $\Sigma_\eps$. Since
$\eps.\delta_I\le\mu\le\eps^{-1}.\delta_I$, we have $\eps I\le X_s\le\eps^{-1}I$ for all
$s\in(0,1]$. Hence $X_s\searrow X_0$ for some $X_0\in\bP$. It remains to show that
$X_0=\cH(\mu)$. For every $s\in(0,1]$ and every $A\in\supp(\mu)$, since
$X_s^{1/2}A^{-1}X_s^{1/2}\le\eps^{-1}X_s\le\eps^{-2}I$,
$\|X_s^{1/2}A^{-1}X_s^{1/2}-I\|\le\eps^{-2}$. Hence we write, as $s\searrow0$,
$$
\bigl[I+s(X_s^{1/2}A^{-1}X_s^{1/2}-I)\bigr]^{-1}=I-s(X_s^{1/2}A^{-1}X_s^{1/2}-I)+o(s),
$$
where $o(s)/s\to0$ in the operator norm as $s\searrow0$ uniformly for $A\in\supp(\mu)$.
Therefore, from \eqref{F-5.1} we find that
$$
I=(1+s)I-s\int_\bP X_s^{1/2}A^{-1}X_s^{1/2}\,d\mu(A)+o(s)\quad\mbox{as $s\searrow0$},
$$
which yields that
$$
I=\lim_{s\searrow0}\int_\bP X_s^{1/2}A^{-1}X_s^{1/2}\,d\mu(A)
$$
in the operator norm. On the other hand, for every $\xi\in\cH$, the bounded convergence
theorem gives
$$
\lim_{s\searrow0}\int_\bP\<\xi,X_s^{1/2}A^{-1}X_s^{1/2}\xi\>\,d\mu(A)
=\int_\bP\<\xi,X_0^{1/2}A^{-1}X_0^{1/2}\xi\>\,d\mu(A).
$$
Therefore, $I=\int_\bP X_0^{1/2}A^{-1}X_0^{1/2}\,d\mu(A)$ so that
$X_0^{-1}=\int_\bP A^{-1}\,d\mu(A)$, i.e., $X_0=\cH(\mu)$.
\end{proof}

Since $\cA_{\,!_s}$ satisfies (vii) by Theorem \ref{T-4.2}, it follows from Proposition
\ref{P-5.1} that $\cH$ satisfies (vii) as well as all other properties in (i)--(viii).

\subsection{Power means}

For each $r\in[-1,1]\setminus\{0\}$ the power mean $P_r$ on $\cPP$ is introduced as the
solution to the equation for $X\in\bP$
\begin{align}\label{F-5.3}
\begin{cases}
X=\cA(X\#_r\mu) & \text{when $r\in(0,1]$}, \\
X=\cH(X\#_{-r}\mu) & \text{when $r\in[-1,0)$},
\end{cases}
\end{align}
that is, in our notation, $P_r=\cA_{\#_r}$ and $P_{-r}=\cH_{\#_r}$ for $r\in(0,1]$.
As mentioned in Remark \ref{R-3.2}, \eqref{F-5.3} is rewritten as a typical case of
the generalized Karcher equation introduced by P\'alfia \cite{Pa}. Among the properties
in (i)--(viii), the only properties not well-known for $P_r$ are (iii) and (vii)
(including (vii-1) and (vii-2)); the other properties are included in
\cite[Theorem 6.4, Proposition 6.15]{Pa}. But, all the properties in (i)--(viii) for
$P_r$ are immediate consequences of Theorems \ref{T-3.1}\,(3) and \ref{T-4.2} applied to
$M=\cA$ or $\cH$, since $\cA$ and $\cH$ satisfies them, as shown in Section 5.1.

\begin{remark}\label{R-5.2}\rm
Note that the $(\dT)_\infty^W$-contractivity of
the power means on $\cP_{cp}(\bP)$, the set of $\mu\in {\mathcal
P}(\bP)$ with compact support, was given in \cite[Proposition
6.7]{KiLaLi}. But it follows from Theorem $\ref{T-2.8}$ that the
power means are $d_\infty^W$-contractive on $\cPP$ bigger than
$\cP_{cp}(\bP)$.
\end{remark}

\subsection{Karcher mean}

The \emph{Karcher mean} (or the \emph{Cartan barycenter}) $G$ on $\cPP$ is introduced as the
solution to the \emph{Karcher equation}
$$
\int_\bP\log X^{-1/2}AX^{-1/2}\,d\mu(A)=0
$$
for given $\mu\in\cPP$, which is the original case of the generalized Karcher equation in
\cite{Pa}. So the properties (i)--(viii) for $G$, except (iii) and (vii), are known in
\cite{Pa}. Below we will prove (iii) and (vii) for $G$ based on the convergence $P_r\to G$
as $r\to0$.

For the $n$-variable weighted case with a weight $\bw=(w_1,\dots,w_n)$, the convergence
$P_{\bw,r}(A_1,\dots,A_n)\to G_\bw(A_1,\dots,A_n)$ as $r\to0$ was first established in
\cite{LP} when $\dim\cH<\infty$ and then extended in \cite{LL} to the SOT-convergence when
$\dim\cH=\infty$. Here note that
$G_\bw(A_1,\dots,A_n)=G\bigl(\sum_{j=1}^nw_j\delta_{A_j}\bigr)$ and
$P_{\bw,r}(A_1,\dots,A_n)=P_r\bigl(\sum_{j=1}^nw_j\delta_{A_j}\bigr)$. For the probability
measure case, the convergence $P_r\to G$ was shown in \cite{KL} when $\dim\cH<\infty$, and
that for compactly supported probability measures when $\dim\cH=\infty$ was in
\cite[Theorem 7.4]{KiLaLi}. In the following we give the convergence for probability measures
in $\cPP$ when $\dim\cH=\infty$. (Even an operator norm convergence of $P_r\to G$ is given
in \cite{LP2}.)

\begin{prop}\label{P-5.2}
For every $\mu\in\cPP$,
$$
P_{-r}(\mu)\le P_{-r'}(\mu)\le G(\mu)\le P_{r'}(\mu)\le P_r(\mu)
\quad\mbox{for $0<r'<r\le1$},
$$
and
$$
G(\mu)=\lim_{r\to0}P_r(\mu)\quad\mbox{in SOT}.
$$
\end{prop}

\begin{proof}
Let $\mu\in\cPP$ and set $X_r:=P_r(\mu)$ for each $r\in(0,1]$. Assume that $0<r'<r\le1$.
Since $X_r=\cA(X_r\#_r\mu)$, we have
\begin{align}\label{F-5.4}
I=\cA(I\#_r(X_r^{-1/2}\mu X_r^{-1/2}))=\int_\bP(X_r^{-1/2}AX_r^{-1/2})^r\,d\mu(A).
\end{align}
By the operator concavity of $x^{r'/r}$ we have
$$
I\ge\int_\bP\bigl[(X_r^{-1/2}AX_r^{-1/2})^r\bigr]^{r'/r}\,d\mu(A)
=\int_\bP(X_r^{-1/2}AX_r^{-1/2})^{r'}\,d\mu(A)
$$
so that $X_r\ge\cA(X_r\#_{r'}\mu)$ and so $X_r\ge X_{r'}$ by Theorem \ref{T-3.1}\,(2).
Therefore, $P_{r'}(\mu)\le P_r(\mu)$, which also implies that $P_{-r}(\mu)\le P_{-r'}(\mu)$
since $P_{-r}(\mu)=P_r(\mu^{-1})^{-1}$.

Choose an $\eps\in(0,1)$ as in the proof of Proposition \ref{P-5.1}. Since $X_r\ge\eps I$,
$X_r\searrow X_0$ for some $X_0\in\bP$. Since $\|X_r^{-1/2}AX_r^{-1/2}\|\le\eps^{-2}$ for
every $r\in(0,1]$ and every $A\in\supp(\mu)$, we have, as $r\searrow0$,
$$
(X_r^{-1/2}AX_r^{-1/2})^r=\exp(r\log X_r^{-1/2}AX^{-1/2})
=I+r\log X_r^{-1/2}AX^{-1/2}+o(r),
$$
where $o(r)/r\to0$ in the operator norm as $r\searrow0$ uniformly for $A\in\supp(\mu)$.
Therefore, from \eqref{F-5.4} we find that
$$
I=I+r\int_\bP\log X_r^{-1/2}AX_r^{-1/2}\,d\mu(A)+o(r)
$$
so that
\begin{align}
\lim_{r\searrow0}\int_\bP\log X_r^{-1/2}AX_r^{-1/2}\,d\mu(A)=0
\end{align}
in the operator norm. On the other hand, note that
$X_r^{-1/2}AX_r^{-1/2}\to X_0^{-1/2}AX_0^{-1/2}$ in SOT and hence
$\log X_r^{-1/2}AX_r^{-1/2}\to\log X_0^{-1/2}AX_0^{-1/2}$ in SOT as $r\searrow0$. For every
$\xi\in\cH$, the bounded convergence theorem gives
$$
\lim_{r\searrow0}\int_\bP\<\xi,(\log X_r^{-1/2}AX_r^{-1/2})\xi\>\,d\mu(A)
=\int_\bP\<\xi,(\log X_0^{-1/2}AX_0^{-1/2})\xi\>\,d\mu(A).
$$
Therefore, $\int_\bP\log X_0^{-1/2}AX_0^{-1/2}\,d\mu(A)=0$ so that $X_0=G(\mu)$. Hence we
have $P_r(\mu)\searrow G(\mu)$ as $r\searrow0$, which also implies that
$P_{-r}(\mu)=P_r(\mu^{-1})^{-1}\nearrow G(\mu^{-1})^{-1}=G(\mu)$ as $r\searrow0$, so that
$P_r(\mu)\to G(\mu)$ in SOT as $r\to0$.
\end{proof}

Since $P_r$ satisfies (vii) as shown in Section 4.2, we see by Proposition \ref{P-5.2} that
$G$ satisfies the same. Moreover, we have

\begin{prop}
The Karcher mean $G$ satisfies {\rm(iii)}.
\end{prop}

\begin{proof}
By Proposition \ref{P-5.2},
$$
G(\mu)=\inf_{0<r\le1}P_r(\mu)=\sup_{0<r\le1}P_{-r}(\mu),\qquad\mu\in\cPP.
$$
Since $P_r$ satisfies (iii) as shown in Section 4.2, we find that if $\mu_k\searrow\mu$ then
for any $\xi\in\cH$,
\begin{align*}
\<\xi,G(\mu)\xi\>&=\inf_{0<r\le1}\<\xi,P_r(\mu)\xi\>
=\inf_{0<r\le1}\inf_{k\ge1}\<\xi,P_r(\mu_k)\xi\> \\
&=\inf_{k\ge1}\inf_{0<r\le1}\<\xi,P_r(\mu_k)\xi\>
=\inf_{k\ge1}\<\xi,G(\mu_k)\xi\>.
\end{align*}
Therefore, $G(\mu_k)\searrow G(\mu)$. If $\mu_k\nearrow\mu$, then we have
$G(\mu_k)\nearrow G(\mu)$ similarly; or since $\mu_k^{-1}\searrow\mu^{-1}$, we have
$G(\mu_k)=G(\mu_k^{-1})^{-1}\nearrow G(\mu^{-1})^{-1}=G(\mu)$.
\end{proof}

In this way, we have seen that all of $\cA,\cH,G$ and $P_r$ for $r\in[-1,1]\setminus\{0\}$
satisfy all the properties in (i)--(viii).

We end the section with an open problem.

\begin{problem}\label{Q-5.4}\rm
Assume that $\mu,\mu_k\in\cPP$ ($k\in\bN$) are supported on $\Sigma_\eps$ for some $\eps>0$
and $\mu_k\to\mu$ weakly on $\Sigma_\eps$ with SOT, i.e.,
$\int_\bP f(A)\,d\mu_k(A)\to\int_\bP f(A)\,d\mu(A)$ for every bounded SOT-continuous real
function $f$ on $\Sigma_\eps$. Can we have $M(\mu_k)\to M(\mu)$ in SOT, for instance, when
$M=G$? This SOT-continuity property is a modification of (iii) without monotonicity
assumption $\mu_k\nearrow$ or $\mu_k\searrow$. The problem was raised in \cite[Section 8]{LL}
for the $n$-variable Karcher mean $G_\bw$: Is the map
$$
(A_1,\dots,A_n)\in(\Sigma_\eps)^n\ \longmapsto\ G_\bw(A_1,\dots,A_n)
=G\Biggl(\sum_{j=1}^nw_j\delta_{A_j}\Biggr)
$$
continuous in SOT? Note here that if $A_j,A_{j,k}\in\Sigma_\eps$
($1\le j\le n$, $k\in\bN$) and $A_{j,k}\to A_j$ in SOT as
$k\to\infty$, then
$\sum_{j=1}^nw_j\delta_{A_{j,k}}\to\sum_{j=1}^nw_j\delta_{A_j}$
weakly on $\Sigma_\eps$ with SOT. Note that the problem is true for
any two-variable operator mean $\sigma$ (in the Kubo-Ando sense),
see Proposition \ref{P-A.1} in Appendix A. For the probability
measure case, we note that the SOT-continuity property stated above
holds for $M=\cA$ and $\cH$, whose proof is not so easy and given in
Appendix A for completeness.
\end{problem}

\section{Applications}

In this section we apply the fixed point method presented in Theorem \ref{T-3.1} to some
important inequalities. To do so, it is convenient to introduce some classes of operator
means on $\cPP$.

\subsection{Derived classes of operator means}
To define some classes operator means on $\cPP$, we consider the following two procedures:
\begin{itemize}
\item[(A)] \emph{Deformation:} from an operator mean $M$ on $\cPP$ (satisfying (i)--(iv))
and a two-variable operator mean (in the Kubo-Ando sense) $\sigma\ne\frak{l}$, define the
deformed operator mean $M_\sigma$. Then $M_\sigma$ is an operator mean on $\cPP$ again by
Theorem \ref{T-3.1}.
\item[(B)] \emph{Composition:} from operator means $M_0,M_1,\dots,M_n$ on $\cPP$ and a weight
$(w_1,\dots,w_n)$ with any $n\in\bN$, define
$M(\mu):=M_0\bigl(\sum_{j=1}^nw_j\delta_{M_j(\mu)}\bigr)$ for $\mu\in\cPP$. Then it is
immediate to see that $M$ is an operator mean on $\cPP$ again.
\end{itemize}

\begin{definition}\label{D-6.1}\rm
We denote by $\fM(\cH)$, or simply $\fM$, the class of operator means on
$\cPP=\cP^\infty(\bP(\cH))$ obtained by starting from $\cA,\cH,G$ (see Sections 4.1 and 4.3)
and applying procedures (A) and (B) finitely many times. We refer to an operator mean $M$ in
the class $\fM$ as a \emph{derived operator mean} on $\cPP$.
\end{definition}

The next proposition says that the derived operator means on $\cP^\infty(\bP(\cH))$ are
defined, in a sense, independently of the choice of (separable) $\cH$.

\begin{prop}\label{P-6.2}
\begin{itemize}
\item[(1)] Assume that $\cH$ is isomorphic to another Hilbert space $\widetilde\cH$ with a
unitary $U:\cH\to\widetilde\cH$. For each $M\in\fM(\cH)$ let $\widetilde M$ be the derived
operator mean on $\cP^\infty(\bP(\widetilde\cH))$ defined by applying {\rm(A)} and {\rm(B)}
in the same way as defining $M$. Then
\begin{align}\label{F-6.1}
\widetilde M(U\mu U^*)=UM(\mu)U^*,\qquad\mu\in\cP^\infty(\bP(\cH)),
\end{align}
where $U\mu U^*$ is the push-forward of $\mu$ by the unitary conjugation
$U\cdot U^*:\bP(\cH)\to\bP(\widetilde\cH)$.
\item[(2)] Assume that $\cH=\cH_1\oplus\cH_2$ with Hilbert spaces $\cH_1,\cH_2$. For each
$M\in\fM(\cH)$ let $M^{(i)}$ be the derived operator mean on $\cP^\infty(\bP(\cH_i))$,
$i=1,2$, defined by applying {\rm(A)} and {\rm(B)} in the same way as defining $M$. Then
\begin{align}\label{F-6.2}
M(\mu_1\oplus\mu_2)=M^{(1)}(\mu_1)\oplus M^{(2)}(\mu_2),
\qquad\mu_i\in\cP^\infty(\bP(\cH_i)),\ i=1,2,
\end{align}
where $\mu_1\oplus\mu_2$ is the push-forward of $\mu_1\times\mu_2$ by the map
$(A,B)\in\bP(\cH_1)\times\bP(\cH_2)\mapsto A\oplus B\in\bP(\cH)$.
In particular,
$$
M(\mu_1\oplus\delta_{I_2})=M^{(1)}(\mu_1)\oplus I_2,\qquad\mu_1\in\cP^\infty(\bP(\cH_1)),
$$
where $I_2$ is the identity operator on $\cH_2$.
\end{itemize}
\end{prop}

\begin{proof}
(1)\enspace
By their definitions (see Sections 4.1 and 4.3) it is immediate to see that $\cA$, $\cH$ and
$G$ satisfy \eqref{F-6.1}. Hence it suffices to show that property \eqref{F-6.1} is preserved
by procedures (A) and (B), which is easily verified and left to the reader.

(2)\enspace
It is convenient for us to consider, in addition to \eqref{F-6.2}, the following property for
$M\in\fM(\cH)$: for any weight $(w_1,\dots,w_m)$, $A_k\in\bP(\cH_1)$ and $B_k\in\bP(\cH_2)$
($1\le k\le m$),
\begin{align}\label{F-6.3}
M\Biggl(\sum_{k=1}^mw_k\delta_{A_k\oplus B_k}\Biggr)
=M^{(1)}\Biggl(\sum_{k=1}^mw_k\delta_{A_k}\Biggr)\oplus
M^{(2)}\Biggl(\sum_{k=1}^mw_k\delta_{B_k}\Biggr).
\end{align}
We now show that \eqref{F-6.2} and \eqref{F-6.3} are preserved by procedures (A) and (B).
Assume that $M\in\fM(\cH)$ satisfies \eqref{F-6.2} and \eqref{F-6.3}, and prove that
$M_\sigma$ does the same for any two-variable operator mean $\sigma\ne\frak{l}$. For
$\mu_i\in\cP^\infty(\bP(\cH_i))$ ($i=1,2$), let $X_i:=M_\sigma^{(i)}(\mu_i)$ so that
$X_i=M^{(i)}(X_i\sigma\mu_i)$. Then
\begin{align*}
X_1\oplus X_2&=M^{(1)}(X_1\sigma\mu_1)\oplus M^{(2)}(X_2\sigma\mu_2) \\
&=M((X_1\sigma\mu_1)\oplus(X_2\sigma\mu_2)) \\
&=M((X_1\oplus X_2)\sigma(\mu_1\oplus\mu_2)),
\end{align*}
where the last equality follows from the well-known property of two-variable operator means
\begin{align}\label{F-6.4}
(X_1\sigma A_1)\oplus(X_2\sigma A_2)=(X_1\oplus X_2)\sigma(A_1\oplus A_2),
\qquad A_i\in\bP(\cH_i).
\end{align}
Hence one has $X_1\oplus X_2=M_\sigma(\mu_1\oplus\mu_2)$ so that $M_\sigma$ satisfies
\eqref{F-6.2}. To prove \eqref{F-6.3} for $M_\sigma$, let
$Y_1:=M_\sigma^{(1)}\bigl(\sum_{k=1}^mw_k\delta_{A_k}\bigr)$ and
$Y_2:=M_\sigma^{(1)}\bigl(\sum_{k=1}^mw_k\delta_{B_k}\bigr)$. Then, by using \eqref{F-6.3}
for $M$ and \eqref{F-6.4}, one has
\begin{align*}
Y_1\oplus Y_2&=M^{(1)}\Biggl(Y_1\sigma\Biggl(\sum_{k=1}^mw_k\delta_{A_k}\Biggr)\Biggr)
\oplus M^{(2)}\Biggl(Y_2\sigma\Biggl(\sum_{k=1}^mw_k\delta_{B_k}\Biggr)\Biggr) \\
&=M^{(1)}\Biggl(\sum_{k=1}^mw_k\delta_{Y_1\sigma A_k}\Biggr)
\oplus M^{(2)}\Biggl(\sum_{k=1}^mw_k\delta_{Y_2\sigma B_k}\Biggr) \\
&=M\Biggl(\sum_{k=1}^mw_k\delta_{(Y_1\oplus Y_2)\sigma(A_k\oplus B_l)}\Biggr)
=M\Biggl((Y_1\oplus Y_2)\sigma\Biggl(\sum_{lk=1}^mw_k\delta_{A_k\oplus B_k}\Biggr)\Biggr),
\end{align*}
which implies that $Y_1\oplus Y_2=M_\sigma\bigl(\sum_{k=1}^mw_k\delta_{A_k\oplus B_k}\bigr)$,
i.e., $M_\sigma$ satisfies \eqref{F-6.3}.

Next, to prove that \eqref{F-6.2} and \eqref{F-6.3} are preserved by procedure (B), assume
that $M_0,M_1,\dots,M_n\in\fM(\cH)$ satisfy them, and let $M$ be given as in (B) with a
weight $(w_1,\dots,w_n)$. For $\mu_i\in\cP^\infty(\bP(\cH_i))$, by using \eqref{F-6.2} for
$M_j$ and \eqref{F-6.3} for $M_0$, one has
\begin{align*}
M(\mu_1\oplus\mu_2)
&=M_0\Biggl(\sum_{j=1}^mw_j\delta_{M_j^{(1)}(\mu_1)\oplus M_j^{(2)}(\mu_2)}\Biggr) \\
&=M_0^{(1)}\Biggl(\sum_{j=1}^nw_j\delta_{M_j^{(1)}(\mu_1)}\Biggr)
\oplus M_0^{(2)}\Biggl(\sum_{j=1}^nw_j\delta_{M_j^{(2)}(\mu_2)}\Biggr) \\
&=M^{(1)}(\mu_1)\oplus M^{(2)}(\mu_2).
\end{align*}
implying \eqref{F-6.2} for $M$. Moreover, for any weight $(w_1',\dots,w_m')$, by using
\eqref{F-6.3} for $M_j$ and $M_0$, one has
\begin{align*}
&M\Biggl(\sum_{k=1}^mw_k'\delta_{A_k\oplus B_k}\Biggr) \\
&\qquad=M_0\Biggl(\sum_{j=1}^nw_j\delta_{M_j^{(1)}\bigl(\sum_{kk=1}^mw_k'\delta_{A_k}\bigr)
\oplus M_j^{(2)}\bigl(\sum_{k=1}^mw_k'\delta_{B_k}\bigr)}\Biggr) \\
&\qquad=M_0^{(1)}\Biggl(\sum_{j=1}^nw_j
\delta_{M_j^{(1)}\bigl(\sum_{k=1}^mw_k'\delta_{A_k}\bigr)}\Biggr)
\oplus M_0^{(2)}\Biggl(\sum_{j=1}^nw_j
\delta_{M_j^{(2)}\bigl(\sum_{k=1}^mw_k'\delta_{B_k}\bigr)}\Biggr) \\
&\qquad=M^{(1)}\Biggl(\sum_{k=1}^mw_k'\delta_{A_k}\Biggr)
\oplus M^{(2)}\Biggl(\sum_{k=1}^mw_k'\delta_{B_k}\Biggr),
\end{align*}
implying \eqref{F-6.3} for $M$. Thus, \eqref{F-6.2} has been shown for all $M\in\fM(\cH)$.
\end{proof}

In what follows, in view of Proposition \ref{P-6.2}, we write a derived operator mean
$M\in\fM$ in common for any choice of the underlying Hilbert space $\cH$.

\begin{prop}\label{P-6.3}
Every operator mean $M\in\frak{M}$ satisfies all the properties in {\rm(i)--(viii)}, and
$M\in\fM$ implies $M^*\in\frak{M}$, where $M^*$ is the adjoint of $M$ (see Proposition
\ref{P-4.1}\,(3)).
\end{prop}

\begin{proof}
First, note that $\cA,\cH,G$ satisfies all (i)--(viii). From definition and Theorems
\ref{T-3.1} and \ref{T-4.2}, to prove the first assertion, it remains to show that the
properties in (v)--(viii) are also preserved under procedure (B). It is immediate to see
this for (v), (vi) and (viii). As for (vii), assume that operator means $M_0,M_1,\dots,M_n$
satisfy (vii), and let $M$ be defined as in (B) with a weight $(w_1,\dots,w_n)$. Let
$\mu_k,\nu_k\in\cPP$ ($1\le k\le m$), and $(w'_1,\dots,w'_m)$ be a weight, and let $0<t<1$.
Then
\begin{align*}
M\Biggl(\sum_{k=1}^mw'_k(\mu_k\triangledown_t\nu_k)\Biggr)
&=M_0\Biggl(\sum_{j=1}^nw_j\delta_{M_j\bigl(\sum_{k=1}^m
w'_k(\mu_k\triangledown_t\nu_k)\bigr)}\Biggr) \\
&\ge M_0\Biggl(\sum_{j=1}^nw_j\delta_{(1-t)M_j\bigl(\sum_{k=1}^mw'_k\mu_k\bigr)
+tM_j\bigl(\sum_{k=1}^mw'_k\nu_k\bigr)}\Biggr) \\
&=M_0\Biggl(\sum_{j=1}^kw_j\delta_{M_j\bigl(\sum_{k=1}^mw'_k\mu_k\bigr)}
\triangledown_t\delta_{M_j\bigl(\sum_{k=1}^mw'_k\nu_k\bigr)}\Biggr) \\
&\ge(1-t)M_0\Biggl(\sum_{j=1}^nw_j\delta_{M_j\bigl(\sum_{k=1}^mw'_k\mu_k\bigr)}\Biggr) \\
&\qquad+tM_0\Biggl(\sum_{j=1}^nw_j\delta_{M_j\bigl(\sum_{k=1}^mw'_k\nu_k\bigr)}\Biggr) \\
&=(1-t)M\Biggl(\sum_{k=1}^mw'_k\mu_k\Biggr)+tM\Biggl(\sum_{k=1}^mw'_k\nu_k\Biggr),
\end{align*}
so that $M$ satisfies (vii) again.

Next, note that the adjoint of $M$ defined in (B) is
$$
M^*(\mu)=M_0^*\Biggl(\sum_{j=1}^nw_j\delta_{M_j^*(\mu)}\Biggr),\qquad\mu\in\cPP.
$$
Since $\cA^*=\cH$ and $G^*=G$, it follows from Proposition \ref{P-4.1}\,(3) and the above
expression that $M\in\frak{M}$ implies $M^*\in\frak{M}$.
\end{proof}

Let $\sigma$ be a two-variable operator mean in the Kubo-Ando sense
with the representing function $f_\sigma$. Following \cite{Wa} we
say that $\sigma$ is \emph{power monotone increasing}
(\emph{p.m.i.}\ for short) if $f_\sigma(x^r)\ge f_\sigma(x)^r$ for
all $x>0$ and $r\ge1$, and \emph{power monotone decreasing}
(\emph{p.m.d.}) if $f_\sigma(x^r)\le f_\sigma(x)^r$ for all $x>0$
and $r\ge1$. We say also that $\sigma$ is \emph{g.c.v.}\ (resp.,
\emph{g.c.c}) if $f_\sigma$ is geometrically convex (resp.,
geometrically concave), i.e.,
$f_\sigma(\sqrt{xy})\le\sqrt{f_\sigma(x)f_\sigma(y)}$ (resp.,
$f_\sigma(\sqrt{xy})\ge\sqrt{f_\sigma(x)f_\sigma(y)}$) for all
$x,y>0$, that is, $\log f(e^t)$ is convex (resp., concave) on
$t\in\bR$. It is clear that $\sigma$ is p.m.i.\ if and only if
$\sigma^*$ is p.m.d., and $\sigma$ is g.c.v.\ if and only if
$\sigma^*$ is g.c.c. Note that $\triangledown_\alpha$ is g.c.v.\ and
$!_\alpha$ is g.c.c.\ for any $\alpha\in[0,1]$. Moreover, the
two-variable operator means that are simultaneously p.m.i.\ and
p.m.d.\ are only $\#_\alpha$ ($0\le\alpha\le1$). It is easy to see
that g.c.v.\ implies p.m.i.\ and g.c.c.\ implies p.m.d.\ for
$\sigma$. But in \cite{Wa2} Wada recently proved that the converse
is not true, that is, there is a p.m.i.\ $\sigma$ that is not g.c.v.

\begin{definition}\label{D-6.4}\rm
(1)\enspace
We denote by $\fM^+$ (resp., $\fM^-$) the subclass of $\fM$ obtained by starting from
$\cA,G$ (resp., $\cH,G$) and applying finitely many times procedure (A) with $\sigma$
restricted to g.c.v.\ (resp., g.c.c.) and procedure (B). Note that $M\in\fM^+$ if and only
if $M^*\in\fM^-$, and $P_r\in\fM^+$ and $P_{-r}\in\fM^-$ for $0<r\le1$ (see Section 4.2).

(2)\enspace
We denote by $\fM_0^+$ (resp., $\fM_0^-$) the subclass of $\fM$ obtained by starting from
$\cA,G$ (resp., $\cH,G$) and applying finitely many times procedure (A) with $\sigma$
restricted to p.m.i.\ (resp., p.m.d.), where procedure (B) is not applied. Note that
$M\in\frak{M}_0^+$ if and only if $M^*\in\frak{M}_0^-$, and $P_r\in\frak{M}_0^+$ and
$P_{-r}\in\frak{M}_0^-$ for $0<r\le1$.
\end{definition}

\subsection{Inequality under positive linear maps}

Let $\cH$ and $\cK$ be separable Hilbert spaces, and let $\Phi:B(\cH)\to B(\cK)$ be a normal
positive linear map. Here, $\Phi$ is normal if $A_k\nearrow A$ in $B(\cH)^+$ implies
$\Phi(A_k)\nearrow\Phi(A)$ in $B(\cK)^+$. Assume that $\Phi(I_\cH)$ is invertible, where
$I_\cH$ is the identity operator on $\cH$. Then $\Phi$ maps $\bP(\cH)$ into $\bP(\cK)$.
In the case where $\cH$ is finite-dimensional, the normality assumption is automatic and
the invertibility assumption of $\Phi(I_\cH)$ is not essential. In fact, let $P_0$ be the
support projection of $\Phi(I_\cH)$; then $\Phi$ may be considered as a map from $B(\cH)$
to $B(P_0\cK)$ and we may replace $\cK$ with $P_0\cK$.

For any two-variable operator mean $\sigma$ and any positive linear map
$\Phi:B(\cH)\to B(\cK)$, the following inequality is well-known:
\begin{align}\label{F-6.5}
\Phi(A\sigma B)\le\Phi(A)\sigma\Phi(B),\qquad A,B\in\bP(\cH),
\end{align}
which is essentially due to Ando \cite{An} while proved only for the geometric and the
harmonic means.

\begin{thm}\label{T-6.5}
Let $\Phi$ be as stated above. The for every derived operator mean $M\in\frak{M}$,
\begin{align}\label{F-6.6}
\Phi(M(\mu))\le M(\Phi_*\mu),\qquad\mu\in\cP^\infty(\bP(\cH)),
\end{align}
where $\Phi_*\mu$ is the push-forward of $\mu$ by the map
$A\in\bP(\cH)\mapsto\Phi(A)\in\bP(\cK)$.
\end{thm}

\begin{proof}
Define $\Psi:B(\cH)\to B(\cK)$ by $\Psi(A):=\Phi(I_\cH)^{-1/2}\Phi(A)\Phi(I_\cH)^{-1/2}$;
then $\Psi$ is a unital positive map. By congruence invariance (vi) (see Section 4),
inequality \eqref{F-6.6} is equivalent to $\Psi(M(\mu))\le M(\Psi_*\mu)$. Hence we may
assume that $\Phi$ is unital. We prove that \eqref{F-6.6} is preserved under procedure (A),
that is, if $M$ satisfies \eqref{F-6.6}, then the deformed operator mean $M_\sigma$ does the
same for any operator mean $\sigma\ne\frak{l}$. For every $\mu\in\cP^\infty(\bP(\cH))$ let
$X_0:=M_\sigma(\mu)$. Then
\begin{align}\label{F-6.7}
\Phi(X_0)=\Phi(M(X_0\sigma\mu))\le M(\Phi_*(X_0\sigma\mu)).
\end{align}
Let $\psi_{X_0}(A):=X_0\sigma A$ for $A\in\bP(\cH)$ and
$\psi_{\Phi(X_0)}(B):=\Phi(X_0)\sigma B$ for $B\in\bP(\cK)$. Note that
$$
\Phi_*(X_0\sigma\mu)=(\Phi\circ\psi_{X_0})_*\mu,\qquad
\Phi(X_0)\sigma(\Phi_*\mu)=(\psi_{\Phi(X_0)}\circ\Phi)_*\mu.
$$
By inequality \eqref{F-6.5} we find that
$$
(\Phi\circ\psi_{X_0})(A)=\Phi(X_0\sigma A)\le\Phi(X_0)\sigma\Phi(A)
=(\psi_{\Phi(X_0)}\circ\Phi)(A)
$$
for all $A\in\bP(\cH)$. Hence it follows from Lemma \ref{L-3.4} that
$$
\Phi_*(X_0\sigma\mu)\le\Phi(X_0)\sigma(\Phi_*\mu).
$$
By monotonicity of $M$ this implies that
\begin{align}\label{F-6.8}
M(\Phi_*(X_0\sigma\mu))\le M(\Phi(X_0)\sigma(\Phi_*\mu)).
\end{align}
By \eqref{F-6.7} and \eqref{F-6.8}, $\Phi(X_0)\le M(\Phi(X_0)\sigma(\Phi_*\mu))$, which
implies by Theorem \ref{T-3.1}\,(2) that $\Phi(X_0)\le M_\sigma(\Phi_*\mu)$.

Next, it is immediate to see that \eqref{F-6.6} is preserved under procedure (B) as follows:
\begin{align*}
\Phi(M(\mu))&\le M_0\Biggl(\Phi_*\Biggl(\sum_{j=1}^nw_j\delta_{M_j(\mu)}\Biggr)\Biggr)
=M_0\Biggl(\sum_{j=1}^nw_j\delta_{\Phi(M_j(\mu))}\Biggr) \\
&\le M_0\Biggl(\sum_{j=1}^nw_j\delta_{M_j(\Phi_*\mu)}\Biggr)=M(\Phi_*\mu).
\end{align*}

Therefore, to prove the theorem, it remains to show that $\cA$, $\cH$ and $G$ satisfy
\eqref{F-6.6}. This is trivial for $\cA$. Apply procedure (A) to $M=\cA$ and $\sigma=\,!_s$
for $0<s\le1$; then we have $\Phi(\cA_{\,!_s}(\mu))\le\cA_{\,!_s}(\Phi_*(\mu))$. Letting
$s\searrow0$ gives $\Phi(\cH(\mu))\le\cH(\Phi_*(\mu))$ thanks to Proposition \ref{P-5.1}
since $\Phi$ is normal. Also, apply (A) to $M=\cA$ and $\sigma=\#_r$ for $0<r\le1$; then
$\Phi(P_r(\mu))\le P_r(\Phi_*\mu)$. Letting $r\searrow0$ gives $\Phi(G(\mu))\le G(\Phi_*\mu)$
thanks to Proposition \ref{P-5.2}.
\end{proof}

\begin{remark}\rm
The normality of $\Phi$ has been used only to prove \eqref{F-6.6}
for $\cH$ and $G$ in the last part of\ the above proof. So, once
$\cH$ and $G$ satisfy \eqref{F-6.6} without the normality assumption
of $\Phi$, we can remove this assumption from Theorem \ref{T-6.5}.
For instance, in view of definition of the power means $P_r$ in
\eqref{F-5.3}, note that $P_r$ for $r\in(0,1]$ satisfies
\eqref{F-6.6} without the normality of $\Phi$. But it is unknown to
us whether this is also the case for $P_r$ for $r\in[-1,0)$, in
particular, for $\cH$. In a different approach in \cite[Theorem
6.4]{Pa} it was shown that \eqref{F-6.6} holds for a certain wide
class of operator means on $\cPP$ under unital positive linear maps
$\Phi$, but the normality of $\Phi$ seems necessary there.
\end{remark}

\subsection{Ando-Hiai's inequality}

In \cite{Wa} Wada proved the extended version of Ando-Hiai's inequality \cite{AH} in such a
way that, for a two-variable operator mean $\sigma$ (in the Kubo-Ando sense), the following
conditions are equivalent:
\begin{itemize}
\item[(i)] $\sigma$ is p.m.i.\ (resp., p.m.d.) (see the paragraph before Definition
\ref{D-6.4});
\item[(ii)] for every $A,B\in\bP$, $A\sigma B\ge I$\,$\implies$\,$A^r\sigma B^r\ge I$ (resp.,
$A\sigma B\le I$\,$\implies$\,$A^r\sigma B^r\le I$) for all $r\ge1$.
\end{itemize}

Ando-Hiai's inequality for the $n$-variable Karcher mean was proved
by Yamazaki \cite{Ya}, which was extended to the case of probability
measures in \cite{KLL,HL} when $\dim\cH<\infty$. In \cite{LY} Lim
and Yamazaki discussed the Ando-Hiai type inequalities for the
$n$-variable power means when $\dim\cH<\infty$. Furthermore, in a
recent paper \cite{HSW} a comprehensive study on the Ando-Hiai type
inequalities for $n$-variable operator means has been made by a
similar fixed point method to this paper. Similar Ando-Hiai type
inequalities have been independently shown by Yamazaki \cite{Ya2}
based on the generalized Karcher equations in \cite{Pa}.

The next theorem is the extension of Ando-Hiai's inequality in \cite{KLL,HL} to the
infinite-dimensional case and to a wider class of operator means on $\cPP$ including power
means, as well as the extension of \cite[Theorem 3.1]{HSW} from $n$-variable operator means
to operator means on $\cPP$.

\begin{thm}\label{T-6.7}
If $M\in\frak{M}_0^+$, then for every $\mu\in\cPP$,
\begin{align}\label{F-6.9}
M(\mu)\ge I\ \implies\ M(\mu^r)\ge I,\ \,r\ge1.
\end{align}
Also, if $M\in\frak{M}_0^-$, then for every $\mu\in\cPP$,
\begin{align}\label{F-6.10}
M(\mu)\le I\ \implies\ M(\mu^r)\le I,\ \,r\ge1.
\end{align}
\end{thm}

\begin{proof}
We show that if $M\in\frak{M}$ satisfies \eqref{F-6.9}, then $M_\sigma$ does the same for
every p.m.i.\ two-variable operator mean $\sigma$. For every $\mu\in\cPP$ assume that $X_0:=M_\sigma(\mu)\ge I$. We
first prove the case where $1\le r\le2$. Since $X_0=M(X_0\sigma\mu)$ so that
$I=M(I\sigma(X_0^{-1/2}\mu X_0^{-1/2}))$, we have
$I\le M\bigl((I\sigma(X_0^{-1/2}\mu X_0^{-1/2}))^r\bigr)$ for all $r\ge1$. Note that
\begin{align}
(I\sigma(X_0^{-1/2}\mu X_0^{-1/2}))^r
&=\bigl(\pi_r\circ f_\sigma\circ\Gamma_{X_0^{-1/2}}\bigr)_*\mu, \label{F-6.11}\\
I\sigma(X_0^{-1/2}\mu^rX_0^{-1/2})
&=\bigl(f_\sigma\circ\Gamma_{X_0^{-1/2}}\circ\pi_r\bigr)_*\mu, \label{F-6.12}
\end{align}
where $\Gamma_{X_0^{-1/2}}(A):=X_0^{-1/2}AX_0^{-1/2}$, $\pi_r(A):=A^r$, and $f_\sigma(A)$ is
the functional calculus of $A$ by the representing function $f_\sigma$. Since
$f_\sigma(x)^r\le f_\sigma(x^r)$ for $x>0$, we find that
\begin{align*}
\bigl(\pi_r\circ f_\sigma\circ\Gamma_{X_0^{-1/2}}\bigr)(A)
&=f_\sigma(X_0^{-1/2}AX_0^{-1/2})^r \\
&\le f_\sigma((X_0^{-1/2}AX_0^{-1/2})^r) \\
&\le f_\sigma(X_0^{-1/2}A^rX_0^{-1/2}) \\
&=\bigl(f_\sigma\circ\Gamma_{X_0^{-1/2}}\circ\pi_r\bigr)(A),\qquad A\in\bP,
\end{align*}
where we have used Hansen-Pedersen's inequality \cite[Theorem 2.1]{HP} for the above latter
inequality thanks to $X_0^{-1/2}\le I$. Applying Lemma \ref{L-3.4} to \eqref{F-6.11} and
\eqref{F-6.12} implies that
$$
(I\sigma(X_0^{-1/2}\mu X_0^{-1/2}))^r\le I\sigma(X_0^{-1/2}\mu^rX_0^{-1/2}).
$$
The monotonicity of $M$ gives
$$
I\le M\bigl((I\sigma(X_0^{-1/2}\mu X_0^{-1/2}))^r\bigr)
\le M(I\sigma(X_0^{-1/2}\mu^rX_0^{-1/2}))
$$
so that $X_0\le M(X_0\sigma\mu^r)$, which implies by Theorem \ref{T-3.1}\,(2) that
$X_0\le M_\sigma(\mu^r)$ and hence $M_\sigma(\mu^r)\ge I$. For general $r\ge1$ write
$r=2^kr_0$ where $k\in\bN$ and $1\le r_0<2$, and iterate the case $1\le r\le2$ to obtain
\eqref{F-6.9} for $M_\sigma$.

In view of Definition \ref{D-6.4}, to prove \eqref{F-6.9} for any $M\in\frak{M}_0^+$, it
remains to show that $\cA$ and $G$ satisfy \eqref{F-6.9}. For $\cA$, when $1\le r\le2$, the
operator convexity of $x^r$ on $(0,\infty)$ gives $\cA(\mu)^r\le\cA(\mu^r)$, so that
\eqref{F-6.9} for $\cA$ holds in this case. The general case $r\ge1$ follows by iteration
as in the last of the first part of the proof. Next, apply the procedure proved in the first
part to $M=\cA$ and $\sigma=\#_\alpha$ for $0<\alpha\le1$; then \eqref{F-6.9} holds for
$P_\alpha$ for $0<\alpha\le1$. Now assume that $G(\mu)\ge I$. For every $\alpha\in(0,1]$,
since $P_\alpha(\mu)\ge G(\mu)\ge I$, we have $P_\alpha(\mu^r)\ge I$ for all $r\ge1$. Letting
$\alpha\searrow0$ gives $G(\mu^r)\ge I$ for all $r\ge1$ thanks to Proposition \ref{P-5.2}.

The latter assertion immediately follows from the first, since \eqref{F-6.9} for $M$ is
equivalent to \eqref{F-6.10} for $M^*$, and $M\in\frak{M}_0^+$ $\iff$ $M^*\in\frak{M}_0^-$.
\end{proof}

\begin{remark}\label{R-6.8}\rm
Note that the proof of Theorem \ref{T-6.7} indeed verifies a slightly stronger result that,
for every $\mu\in\cPP$, if $M\in\frak{M}_0^+$ then
$$
M(\mu)\ge I\ \implies\ M(\mu^r)\ge M(\mu),\ \ r\ge1,
$$
and if $M\in\frak{M}_0^-$ then
$$
M(\mu)\le I\ \implies\ M(\mu^r)\le M(\mu),\ \ r\ge1.
$$
These are equivalently stated in such a way that, for every $\mu\in\cPP$, if
$M\in\frak{M}_0^+$ then
$$
M(\mu^r)\ge\lambda_{\min}(M(\mu))^{r-1}M(\mu),\qquad r\ge1,
$$
and if $M\in\frak{M}_0^-$ then
$$
M(\mu^r)\le\|M(\mu)\|^{r-1}M(\mu),\qquad r\ge1.
$$
In the above, $\lambda_{\min}(A)$ is the minimum of the spectrum of $A\in\bP$.
\end{remark}

\begin{cor}\label{C-6.9}
Let $\alpha\in(0,1]$. Then for every $\mu\in\cPP$,
\begin{align}
P_\alpha(\mu)\ge I\ &\implies\ P_\alpha(\mu^r)\ge I,\ \ r\ge1, \label{F-6.13}\\
P_{-\alpha}(\mu)\le I\ &\implies\ P_{-\alpha}(\mu^r)\le I,\ \ r\ge1. \label{F-6.14}
\end{align}
\end{cor}

\begin{remark}\label{R-6.10}\rm
Corollary \ref{C-6.9} has been shown based on Theorem \ref{T-3.1} applied to the simple case
$M=\cA$. Since $P_\alpha\searrow G$ and $P_{-\alpha}\nearrow G$ as $r\searrow0$ by
Proposition \ref{P-5.2}, the corollary in turn gives Ando-Hiai's inequality for $G$. This is
a new proof of Ando-Hiai's inequality even for the two-variable geometric mean.
\end{remark}

\subsection{Modified Ando-Hiai's inequalities}

We here present two more Ando-Hiai type inequalities, which are weaker than Theorem
\ref{T-6.7} (also Remark \ref{R-6.8}) in the sense that inequalities are between two
deformed operator means $M_\sigma$ and $M_{\sigma_r}$ on $\cPP$, where $0<r\le1$ and
$\sigma_r$ is the modified two-variable operator mean with the representing function
$f_\sigma(x^r)$. But instead, those have an advantage since there are no restrictions on
$M$ (except congruence invariance) and $\sigma$ unlike in Theorem \ref{T-6.7}.

\begin{thm}\label{T-6.11}
Let $M$ be an operator mean on $\cPP$ satisfying {\rm(vi)} (congruence invariance) in
Section 4, and $\sigma$ be any two-variable operator mean with $\sigma\ne\frak{l}$. Then
for every $\mu\in\cPP$,
\begin{align}
&M_\sigma(\mu)\ge I\ \implies\ M_{\sigma_{1/r}}(\mu^r)\ge M_\sigma(\mu),
\ \ r\ge1, \label{F-6.15} \\
&M_\sigma(\mu)\le I\ \implies\ M_{\sigma_{1/r}}(\mu^r)\le M_\sigma(\mu),
\ \ r\ge1. \label{F-6.16}
\end{align}
\end{thm}

\begin{proof}
Let $\mu\in\cPP$ and $r\ge1$. The assertion in \eqref{F-6.16} follows from \eqref{F-6.15} by
replacing $M$, $\sigma$ and $\mu$ in \eqref{F-6.15} with $M^*$, $\sigma^*$ and $\mu^{-1}$.
Indeed, note that $M^*$ satisfies (vi) too and $(\sigma^*)_{1/r}=(\sigma_{1/r})^*$ so that
$(M^*)_{(\sigma^*)_{1/r}}=(M_{\sigma_{1/r}})^*$ as well as $M^*_{\sigma^*}=(M_\sigma)^*$ by
Proposition \ref{P-4.1}\,(3). Hence \eqref{F-6.16} is equivalent to \eqref{F-6.15} for
$M^*$, $\sigma^*$ and $\mu^{-1}$. So we may prove \eqref{F-6.15} only. Assume that
$X_0:=M_\sigma(\mu)\ge I$; then $I=M(I\sigma(X_0^{-1/2}\mu X_0^{-1/2}))$. Using the same
notations as in the proof of Theorem \ref{T-6.7} we have
$I\sigma(X_0^{-1/2}\mu X_0^{-1/2})=\bigl(f_\sigma\circ\Gamma_{X_0^{-1/2}}\bigr)_*\mu$ and
$$
I\sigma_{1/r}(X_0^{-1/2}\mu^rX_0^{-1/2})
=\bigl(f_{\sigma}\circ\pi_{1/r}\circ\Gamma_{X_0^{-1/2}}\circ\pi_r\bigr)_*\mu.
$$
We have
\begin{align*}
\bigl(f_{\sigma}\circ\pi_{1/r}\circ\Gamma_{X_0^{-1/2}}\circ\pi_r\bigr)(A)
&=f_\sigma((X_0^{-1/2}A^rX_0^{-1/2})^{1/r}) \\
&\ge f_\sigma(X_0^{-1/2}AX_0^{-1/2}) \\
&=\bigl(f_\sigma\circ\Gamma_{X_0^{-1/2}}\bigr)(A),\qquad A\in\bP,
\end{align*}
where Hansen's inequality \cite{Ha} has been used for the above inequality since
$X_0^{-1/2}\le I$ and $0<1/r\le1$. Therefore, Lemma \ref{L-3.4} gives
$$
I\sigma(X_0^{-1/2}\mu X_0^{-1/2})\le I\sigma_{1/r}(X_0^{-1/2}\mu^r X_0^{-1/2}),
$$
so that
$$
I=M(I\sigma(X_0^{-1/2}\mu X_0^{-1/2}))\le M(I\sigma_{1/r}(X_0^{-1/2}\mu^r X_0^{-1/2})).
$$
This implies that $X_0\le M(X_0\sigma_{1/r}\mu^r)$ so that $X_0\le M_{\sigma_{1/r}}(\mu^r)$
by Theorem \ref{T-3.1}\,(2).
\end{proof}

\begin{remark}\label{R-6.12}\rm
Similarly to Remark \ref{R-6.8} the assertions \eqref{F-6.15} and \eqref{F-6.16} together
are equivalently stated as
\begin{align}\label{F-6.17}
\lambda_{\min}^{r-1}(M_\sigma(\mu))M_\sigma(\mu)
\le M_{\sigma_{1/r}}(\mu^r)\le\|M_\sigma(\mu)\|^{r-1}M_\sigma(\mu),
\qquad r\ge1.
\end{align}
This is the extension of \cite[Theorem 4.1]{HSW} from the
$n$-variable case to the probability measure case. Note that the
special case of \eqref{F-6.17} for $n$-variable power means was
first shown in \cite[Corollary 3.2]{LY} when $\dim\cH<\infty$.
\end{remark}

The next result is the complementary version of \eqref{F-6.17}, which is the extension of
\cite[Theorem 4.2]{HSW} from the $n$-variable case to the probability measure case.

\begin{thm}\label{T-6.13}
Let $M$ and $\sigma$ be as in Theorem \ref{T-6.11}. For every $\mu\in\cPP$,
\begin{align}\label{F-6.18}
\|M_{\sigma_r}(\mu)\|^{r-1}M_{\sigma_r}(\mu)\le M_\sigma(\mu^r)
\le\lambda_{\min}^{r-1}(M_{\sigma_r}(\mu))M_{\sigma_r}(\mu),
\qquad0<r\le1.
\end{align}
\end{thm}

\begin{proof}
As in the proof of Theorem \ref{T-6.11} the first and the second inequalities in
\eqref{F-6.18} are equivalent by replacing $M$, $\sigma$ and $\mu$ with $M^*$, $\sigma^*$
and $\mu^{-1}$, so we may prove the second only. Let $\mu\in\cPP$ and $0<r\le1$. Moreover,
let $X_0:=M_{\sigma_r}(\mu)$ and $\lambda:=\lambda_{\min}(X_0)$. We have
$I=M(I\sigma_r(X_0^{-1/2}\mu X_0^{-1/2}))$ and
$$
I\sigma_r(X_0^{-1/2}\mu X_0^{-1/2})
=\bigl(f_\sigma\circ\pi_r\circ\Gamma_{X_0^{-1/2}}\bigr)_*\mu.
$$
By Theorem \ref{T-3.1}\,(2) what we need to prove is that
$\lambda^{r-1}X_0\ge M((\lambda^{r-1}X_0)\sigma\mu^r)$, or equivalently,
$I\ge M(I\sigma(\lambda^{1-r}.(X_0^{-1/2}\mu^rX_0^{-1/2})))$. Note that
$$
I\sigma(\lambda^{1-r}.X_0^{-1/2}\mu^rX_0^{-1/2})
=\bigl(f_\sigma\circ m_{\lambda^{1-r}}\circ\Gamma_{X_0^{-1/2}}\circ\pi_r\bigr)_*\mu,
$$
where $m_{\lambda^{1-r}}(A):=\lambda^{1-r}A$. Since $\lambda^{1/2}X_0^{-1/2}\le I$, Hansen's
inequality \cite{Ha} gives
$$
\lambda^{1-r}X_0^{-1/2}A^rX_0^{-1/2}\le(X_0^{-1/2}AX_0^{-1/2})^r,
$$
which implies that
\begin{align*}
\bigl(f_\sigma\circ m_{\lambda^{1-r}}\circ\Gamma_{X_0^{-1/2}}\circ\pi_r\bigr)(A)
&=f_\sigma(\lambda^{1-r}X_0^{-1/2}A^rX_0^{-1/2}) \\
&\le f_\sigma((X_0^{-1/2}AX_0^{-1/2})^r) \\
&=\bigl(f_\sigma\circ\pi_r\circ\Gamma_{X_0^{-1/2}}\bigr)(A),\qquad A\in\bP.
\end{align*}
Therefore, by Lemma \ref{L-3.4} we have
$$
I=M(I\sigma_r(X_0^{-1/2}\mu X_0^{-1/2}))
\ge M(I\sigma(\lambda^{1-r}.(X_0^{-1/2}\mu^rX_0^{-1/2}))),
$$
as required.
\end{proof}

\begin{remark}\label{R-6.14}\rm
Similarly to the case $\dim\cH<\infty$ in \cite[Theorem 3.4]{HL2}, it is easy to verify that
the Karcher mean $G$ satisfies $G(\mu)=G_{\#_\alpha}(\mu)$ for all $\mu\in\cPP$ and any
$\alpha\in(0,1]$. Since $(\#_\alpha)_r=\#_{\alpha r}$ for all $r\in(0,1]$, either Theorem
\ref{T-6.11} or \ref{T-6.13} gives Ando-Hiai's inequality for $G$, in addition to an
demonstration in Remark \ref{R-6.10}.
\end{remark}

\section{Further applications}

In this section we present more applications of our method to different inequalities.
Throughout the section, for some theoretical and technical reasons, we assume that $\cH$ is
\emph{finite-dimensional}.

\subsection{Norm inequality}

For every unitarily invariant norm $|||\cdot|||$ and any two-variable operator mean $\sigma$,
it is well-known (see, e.g., \cite[(3.13)]{An2}) that
\begin{align}\label{F-7.1}
|||A\sigma B|||\le|||A|||\sigma|||B|||,\qquad A,B\in B(\cH)^+.
\end{align}
In fact, this norm inequality can be extended to more general norms $|||\cdot|||$ on
$B(\cH)$ that is \emph{monotone} in the sense that if $A\ge B\ge0$ implies
$|||A|||\ge|||B|||$. There are many examples of monotone norms on $B(\cH)$ that are not
unitarily invariant; for example, the numerical radius and $|||A|||=\|A\|+|\tr\,A|$,
where $\tr\,A$ is the trace of $A$, are such cases. (These are examples of
\emph{weakly unitarily invariant} norms introduced in \cite{BhHo}.)
Although the extension of \eqref{F-7.1} to monotone norms may be folklore to experts,
there seems no literature, so we prove it as a lemma.

\begin{lemma}\label{L-7.1}
If $|||\cdot|||$ is a monotone norm on $B(\cH)$, then inequality \eqref{F-7.1} holds for
every two-variable operator mean $\sigma$.
\end{lemma}

\begin{proof}
By continuity we may assume that $A,B\in B(\cH)^+$ are invertible. Since
$$
(I+A^{1/2}B^{-1}A^{1/2})^{-1}\le t^2I+(1-t)^2A^{-1/2}BA^{-1/2},
$$
we have $A:B\le t^2 A+(1-t)^2B$ for all $t\in\bR$, where $A:B:=(A^{-1}+B^{-1})^{-1}$, the
\emph{parallel sum} of $A,B$. Therefore,
$$
|||A:B|||\le t^2|||A|||+(1-t)^2|||B|||.
$$
Minimizing the right-hand side above gives $|||A:B|||\le|||A|||:|||B|||$. Recall the integral
expression \cite{KA}
\begin{align}\label{F-7.2}
A\sigma B=aA+bB+\int_{(0,\infty)}{1+t\over t}\{(tA):B\}\,dm(t)
\end{align}
with a probability measure $m$ on $[0,\infty]$ where $a:=m(\{0\})$, $b:=m(\{\infty\})$. From
this integral expression we have inequality \eqref{F-7.1} as
\begin{align*}
|||A\sigma B|||&\le a|||A|||+b|||B|||
+\int_{(0,\infty)}{1+t\over t}\{(t|||A|||):|||B|||\}\,dm(t) \\
&=|||A|||\sigma|||B|||.
\end{align*}
\end{proof}

The following is the extension of \eqref{F-7.1} to derived operator means $M\in\fM$.

\begin{prop}\label{P-7.2}
For every derived operator mean $M\in\fM$ and every monotone norm $|||\cdot|||$ on $B(\cH)$,
\begin{align}\label{F-7.3}
|||M(\mu)|||\le M(|||\cdot|||_*\mu),\qquad\mu\in\cPP,
\end{align}
where $|||\cdot|||_*\mu$ is the push-forward of $\mu$ by $A\in\bP\mapsto|||A|||\in(0,\infty)$
(so $|||\cdot|||_*\mu$ is a Borel probability measure on $(0,\infty)$).
\end{prop}

\begin{proof}
We show that \eqref{F-7.3} is preserved under procedure (A). Assume that $M\in\fM$ satisfies
\eqref{F-7.3} and $\sigma$ is a two-variable operator mean with $\sigma\ne\frak{l}$. For
every $\mu\in\cPP$ let $X_0:=M_\sigma(\mu)$. Then
\begin{align}\label{F-7.4}
|||X_0|||=|||M(X_0\sigma\mu)|||\le M(|||\cdot|||_*(X_0\sigma\mu)).
\end{align}
Note that
$$
|||\cdot|||_*(X_0\sigma\mu)=(|||\cdot|||\circ\psi_{X_0})_*\mu,\qquad
|||X_0|||\sigma(|||\cdot|||_*\mu)=(\psi_{|||X_0|||}\circ|||\cdot|||)_*\mu,
$$
where $\psi_{X_0}$ is as in the proof of Theorem \ref{T-6.5} and $\psi_{|||X_0|||}$ is
similarly defined on $(0,\infty)$. It follows from Lemma \ref{L-7.1} that
$$
(|||\cdot|||\circ\psi_{X_0})(A)=|||X_0\sigma A|||\le|||X_0|||\sigma|||A|||
=(\psi_{|||X_0|||}\circ|||\cdot|||)(A)
$$
for all $A\in\bP$, so that $|||\cdot|||_*(X_0\sigma\mu)\le|||X_0|||\sigma(|||\cdot|||_*\mu)$
by Lemma \ref{L-3.4}. From this and \eqref{F-7.4} one has
$$
|||X_0|||\le M(|||X_0|||\sigma(|||\cdot|||_*\mu)).
$$
By Theorem \ref{T-3.1}\,(2) (for the case $\dim\cH=1$) this implies that
$|||X_0|||\le M_\sigma(|||\cdot|||_*\mu)$. The remaining proof is similar to the second and
the third paragraphs of the proof of Theorem \ref{T-6.5}, so the details may be left to the
reader.
\end{proof}

In particular, for the Karcher mean $G$ and for any monotone norm $|||\cdot|||$ on $B(\cH)$,
we have
\begin{align}\label{F-7.5}
|||G(\mu)|||\le G(|||\cdot|||_*\mu)=\exp\int_\bP\log|||A|||\,d\mu(A),
\qquad\mu\in\cPP,
\end{align}
as verified in \cite{Li}, where the last equality is readily seen.

\subsection{Eigenvalue majorizations}

Assume that $N=\dim\cH$. Let us first recall the notion of majorization. Let
$a=(a_1,\dots,a_N)$, $b=(b_1,\dots,b_N)\in(\bR^+)^N$. We write
$a^\downarrow=(a_1^\downarrow,\dots,a_N^\downarrow)$ for the decreasing rearrangement of $a$
and $a^\uparrow=(a_1^\uparrow,\dots,a_N^\uparrow)$ for its increasing counterpart. The
\emph{weak majorization} (or \emph{submajorization}) $a\prec_wb$ means that
$$
\sum_{i=1}^ka_i^\downarrow\le\sum_{i=1}^kb_i^\downarrow,\qquad1\le k\le N.
$$
We call the \emph{majorization} $a\prec b$ if $a\prec_wb$ and equality holds for $k=N$ above.
The \emph{log-majorization} $a\prec_{\log}b$ is defined as
$$
\prod_{i=1}^ka_i^\downarrow\le\prod_{i=1}^kb_i^\downarrow,\quad1\le k\le N-1,
\quad\mbox{and}\quad
\prod_{i=1}^Na_i^\downarrow=\prod_{i=1}^Nb_i^\downarrow.
$$
Note that $a\prec_{\log}b$ $\implies$ $a\prec_wb$.

Let $A,B\in\bP$ and $\lambda(A)=(\lambda_1(A),\dots,\lambda_N(A))$ be the eigenvalues of $A$
in decreasing order counting multiplicities. In the following, we identify $\lambda(A)$ in
$(0,\infty)^N$ with a diagonal matrix in $\bP$ with $\lambda_1(A),\dots\lambda_N(A)$ on its
diagonal. We write $A\prec_wB$, $A\prec B$ and $A\prec_{\log}B$ if each majorization
for $\lambda(A)$ and $\lambda(B)$ holds, respectively. It is well-known that if $A\prec_wB$
then $|||f(A)|||\le|||f(B)|||$ for every unitarily invariant norm $|||\cdot|||$ and every
non-negative and non-decreasing convex function $f$ on $(0,\infty)$. See, e.g.,
\cite{An2,Hi,MOA} for more about majorizations for matrices.

The \emph{Log-Euclidean mean} of $\mu\in\cPP$ is given as
$$
LE(\mu):=\exp\int_\bP\log A\,d\mu(A).
$$
Recall the \emph{Lie-Trotter formula} given in \cite[Theorem 5.7]{HL}
\begin{align}\label{F-7.6}
\lim_{r\searrow0}G(\mu^r)^{1/r}=LE(\mu).
\end{align}
For $\mu\in\cPP$ let $\lambda_*\mu$ denote the push-forward of $\mu$ by the continuous map
$A\in\bP\mapsto\lambda(A)\in(0,\infty)^N$, that is,
$$
\lambda_*\mu=((\lambda_1)_*\mu,\dots,(\lambda_N)_*\mu),
$$
where $(\lambda_i)_*\mu$ is the push-forward of $\mu$ by
$A\in\bP\mapsto\lambda_i(A)\in(0,\infty)$. It is readily seen that
\begin{align}\label{F-7.7}
G(\lambda_*\mu)=LE(\lambda_*\mu)
=\Bigl(\exp\int_\bP\log\lambda_i(A)\,d\mu(A)\Bigr)_{i=1}^N.
\end{align}

\begin{prop}\label{P-7.3}
For every $\mu\in\cPP$ and every $r\in(0,1)$,
$$
G(\mu)\prec_{\log}G(\mu^r)^{1/r}\prec_{\log}LE(\mu)\prec_{\log}G(\lambda_*\mu).
$$
\end{prop}

\begin{proof}
It was proved in \cite[Theorem 4.4]{HL} that
$$
G(\mu)\prec_{\log}G(\mu^r)^{1/r}\prec_{\log}G(\mu^{r'})^{1/r'}
\quad\mbox{for $0<r'<r<1$}.
$$
Combining with \eqref{F-7.6} we have
$$
G(\mu)\prec_{\log}G(\mu^r)^{1/r}\prec_{\log}LE(\mu)
\quad\mbox{for $0<r<1$}.
$$
It remains to prove that $G(\mu^r)^{1/r}\prec_{\log}G(\lambda_*\mu)$ for any $r\in(0,1)$.
Apply Proposition \ref{P-7.2} to $M=G$ and $\|\cdot\|=\lambda_1(\cdot)$; then we have
\begin{align}\label{F-7.8}
\lambda_1(G(\mu))\le G((\lambda_1)_*\mu).
\end{align}
Let $(\wedge^k)_*\mu$ be the push-forward of $\mu$ by the antisymmetric tensor power map
$A\in\bP(\cH)\mapsto\wedge^kA\in\bP(\wedge^k\cH)$, see \cite{Bh,MOA}. For each $k=1,\dots,N$, use
\cite[Theorem 4.2]{HL} and apply \eqref{F-7.8} to $\wedge^k\mu$ to obtain
\begin{align*}
\prod_{i=1}^k\lambda_i(G(\mu))
&=\lambda_1(\wedge^kG(\mu))=\lambda_1(G((\wedge^k)_*\mu)) \\
&\le G((\lambda_1)_*((\wedge^k)_*\mu))
=G\bigg(\Bigl(\prod_{i=1}^k\lambda_i\Bigr)_*\mu\biggr) \\
&=\exp\int_\bP\log\prod_{i=1}^k\lambda_i(A)\,d\mu(A)
=\prod_{i=1}^kG((\lambda_i)_*\mu).
\end{align*}
Moreover,
$$
\prod_{i=1}^N\lambda_i(G(\mu))=\det G(\mu)=\exp\int_\bR\log\det A\,d\mu(A)
=\prod_{i=1}^NG((\lambda_i)_*\mu).
$$
We therefore have
\begin{align}\label{F-7.9}
G(\mu)\prec_{\log}G(\lambda_*\mu).
\end{align}
By replacing $\mu$ with $\mu^r$ for $0<r<1$ we have
$$
G(\mu^r)^{1/r}\prec_{\log}G(\lambda_*(\mu^r))^{1/r}=G(\lambda_*\mu),
$$
where the last equality immediately follows from \eqref{F-7.7}.
\end{proof}

\begin{cor}
Assume that $f$ is a non-negative and non-decreasing function $f$ on $(0,\infty)$ such that
$f(e^x)$ is convex on $\bR$. For every $\mu\in\cPP$ and any unitarily invariant norm
$|||\cdot|||$,
\begin{align}\label{F-7.10}
|||f(G(\mu))|||\le|||f(LE(\mu))|||\le|||f(G(\lambda_*\mu))|||.
\end{align}
In particular,
\begin{align}\label{F-7.11}
|||G(\mu)|||\le|||LE(\mu)|||\le|||G(\lambda_*\mu)|||\le\exp\int_\bP\log|||A|||\,d\mu(A).
\end{align}
\end{cor}

\begin{proof}
Proposition \ref{P-7.3} implies \eqref{F-7.10} by \cite[Proposition 4.1.6]{Hi}. For the last
inequality of \eqref{F-7.11}, since $|||\lambda(A)|||=|||A|||$ for all $A\in\bP$, we have by
Proposition \ref{P-7.2}
\begin{align*}
|||G(\lambda_*\mu)|||&\le G(|||\cdot|||_*(\lambda_*\mu))
=G((|||\cdot|||\circ\lambda)_*\mu) \\
&=G(|||\cdot|||_*\mu)=\exp\int_\bP\log|||A|||\,d\mu(A).
\end{align*}
\end{proof}

Note that \eqref{F-7.11} contains \eqref{F-7.5} for unitarily invariant norms $|||\cdot|||$,
while the latter holds for more general monotone norms.

We have the integral version of the Ky Fan majorization as
\begin{align}\label{F-7.12}
\cA(\mu)\prec\cA(\lambda_*\mu),
\end{align}
i.e.,
$$
\sum_{i=1}^k\lambda_i\biggl(\int_\bP A\,d\mu(A)\biggr)
\le\sum_{i=1}^k\int_\bP\lambda_i(A)\,d\mu(A),\qquad1\le k\le N,
$$
with equality for $k=N$. A little argument by replacing $\mu$ in \eqref{F-7.12} with
$\mu^{-1}$ gives the weak majorization
\begin{align}\label{F-7.13}
\cH(\mu)\prec_w\cH(\lambda_*\mu).
\end{align}
For any two-variable operator mean $\sigma$ and every $A,B\in\bP$ we have
\begin{align}\label{F-7.14}
\lambda(A\sigma B)\prec_w\lambda(A)\sigma\lambda(B).
\end{align}
Since we find no literatue on this, we give the proof in Appendix B.1 for completeness,
together with the proofs of \eqref{F-7.12} and \eqref{F-7.13} in Appendix B.2. In view of
\eqref{F-7.11}--\eqref{F-7.13} as well as \eqref{F-7.14}, one may expect the weak
majorization $\lambda(M(\mu))\prec_w M(\lambda_*\mu)$ for other operator means $M$ on
$\cPP$. In the next proposition we prove that this holds true for the power means $P_r$ with
$0<r\le1$.

\begin{prop}
For every $r\in(0,1]$ and every $\mu\in\cPP$,
\begin{align}\label{F-7.15}
P_r(\mu)\prec_w P_r(\lambda_*\mu).
\end{align}
\end{prop}

\begin{proof}
Since $P_1=\cA$, the case $\alpha=1$ holds by \eqref{F-7.12}. So we may assume that
$0<\alpha<1$. Since a simple calculation gives
$$
P_\alpha(\lambda_*\mu)
=\biggl(\biggl[\int_\bP\lambda_i(A)^\alpha\,d\mu(A)\biggr]^{1/\alpha}\biggr)_{i=1}^N,
$$
what we need to prove is
\begin{align}\label{F-7.16}
\sum_{i=1}^k\lambda_i(P_\alpha(\mu))\le
\sum_{i=1}^k\biggl[\int_\bP\lambda_i(A)^\alpha\,d\mu(A)\biggr]^{1/\alpha},
\qquad1\le k\le N.
\end{align}
Set $X_0:=P_\alpha(\mu)$. Since $X_0=\cA(X_0\#_\alpha\mu)$, for $1\le k\le N$ we have by \eqref{F-7.12}
$$
\sum_{i=1}^k\lambda_i(X_0)\le\sum_{i=1}^k\cA((\lambda_i)_*(X_0\#_\alpha\mu))
=\int_\bP\sum_{i=1}^k\lambda_i(X_0\#_\alpha A)\,d\mu(A).
$$
Since $\lambda(X_0)\#_\alpha\lambda(A)
=(\lambda_i(X_0)^{1-\alpha}\lambda_i(A)^\alpha)_{i=1}^N$, it follows from \eqref{F-7.14} for
$\sigma=\#_\alpha$ that
$$
\sum_{I=1}^k\lambda_i(X_0\#_\alpha A)
\le\sum_{i=1}^k\lambda_i(X_0)^{1-\alpha}\lambda_i(A)^\alpha,\qquad A\in\bP.
$$
Therefore,
$$
\sum_{i=1}^k\lambda_i(X_0)\le
\sum_{i=1}^k\lambda_i(X_0)^{1-\alpha}\int_\bP\lambda_i(A)^\alpha\,d\mu(A).
$$
Thanks to H\"older's inequality we find that
$$
\sum_{i=1}^k\lambda_i(X_0)\le
\Biggl[\sum_{i=1}^k\lambda_i(X_0)\Biggr]^{1-\alpha}
\Biggl(\sum_{i=1}^k\biggl[\int_\bP\lambda_i(A)^r\,d\mu(A)
\biggr]^{1/\alpha}\Biggr)^\alpha,
$$
which implies \eqref{F-7.16}.
\end{proof}

\begin{cor}
Assume that $f$ is a non-negative and non-decreasing convex function on $(0,\infty)$. Let
$\alpha\in(0,1]$. For every $\mu\in\cPP$ and any unitarily invariant norm $|||\cdot|||$,
$$
|||f(P_\alpha(\mu))|||\le|||f(P_\alpha(\lambda_*\mu))|||.
$$
In particular,
$$
|||P_\alpha(\mu)|||\le|||P_\alpha(\lambda_*\mu)|||\le P_\alpha(|||\cdot|||_*\mu)
=\biggl[\int_\bP|||A|||^\alpha\,d\mu(A)\biggr]^{1/\alpha}.
$$
\end{cor}

By letting $\alpha\searrow0$ in \eqref{F-7.15} we have
$$
G(\mu)\prec_w G(\lambda_*\mu),
$$
which is however weaker than \eqref{F-7.9}.

\begin{problem}\rm
When $\alpha=-1$, \eqref{F-7.15} reduces to \eqref{F-7.13}. But it is unknown whether the
weak majorization in \eqref{F-7.15} holds even for $\alpha\in(-1,0)$ or not. The problem is
to prove \eqref{F-7.16} for $\alpha\in(-1,0)$ and $\mu\in\cPP$. Note that a weaker version
$$
\sum_{i=1}^k\lambda_i(P_\alpha(\mu))\le
\biggl[\int_\bP\biggl(\sum_{i=1}^k\lambda_i(A)\biggr)^\alpha\,d\mu(A)
\biggr]^{1/\alpha}
$$
is the special case of the norm inequality in \eqref{F-7.3} for the Ky Fan $k$-norm. Now,
let us try to apply \eqref{F-7.15} to $-\alpha\in(0,1)$ and $\mu^{-1}$; then we have
$$
P_{-\alpha}(\mu^{-1})\prec_wP_{-\alpha}(\lambda_*(\mu^{-1})).
$$
Note that $P_{-\alpha}(\mu^{-1})=P_\alpha(\mu)^{-1}$ and
\begin{align*}
P_{-\alpha}(\lambda_*(\mu^{-1}))&=P_{-\alpha}(((\lambda^\uparrow)_*\mu)^{-1})
=P_\alpha((\lambda^\uparrow)_*\mu)^{-1} \\
&=(P_\alpha(\lambda_*\mu)^\uparrow)^{-1}
=(P_\alpha(\lambda_*\mu)^{-1})^\downarrow,
\end{align*}
where $\lambda^\uparrow(A):=\lambda(A)^\uparrow$. Therefore,
\begin{align}\label{F-7.17}
P_\alpha(\mu)^{-1}\prec_wP_\alpha(\lambda_*\mu)^{-1},
\end{align}
which is a complementary version of \eqref{F-7.15}. Another majorization that is stronger
than \eqref{F-7.17} is the supermajorization
$$
P_\alpha(\mu)\prec^wP_\alpha(\lambda_*\mu),
$$
i.e.,
$$
\sum_{i=1}^k\lambda_{N+1-i}(P_\alpha(\mu))\ge\sum_{i=1}^kP_\alpha((\lambda_{N+1-i})_*\mu),
\qquad1\le k\le N,
$$
which however cannot hold. Indeed, if the above supermajorization holds for $P_{-1}=\cH$,
then together with \eqref{F-7.13} we have $\lambda(\cH(\mu))\prec\cH(\lambda_*\mu)$, which
is impossible.
\end{problem}

\subsection{Minkowski determinant inequality}

The famous \emph{Minkowski determinant inequality} says that, for every $A,B\in B(\cH)^+$,
$$
{\det}^{1/N}(A+B)\ge{\det}^{1/N}A+{\det}^{1/N}B,
$$
or equivalently, $A\in B(\cH)^+\mapsto{\det}^{1/N}A$ is concave:
$$
{\det}^{1/N}(A\triangledown_tB)
\ge({\det}^{1/N}A)\triangledown_t({\det}^{1/N}B),\qquad0\le t\le1.
$$
The following extension was given in \cite[Corollary 3.2]{BH}.

\begin{lemma}\label{L-7.8}
If a two-variable operator mean $\sigma$ (in the Kubo-Ando sense) is
g.c.v.\ (see Section 5.1), then
$$
{\det}^{1/N}(A\sigma B)\ge({\det}^{1/N}A)\sigma({\det}^{1/N}B),\qquad A,B\in B(\cH)^+.
$$
The reverse inequality holds if $\sigma$ is g.c.c.
\end{lemma}

The following is the further extension to derived operator means $M\in\fM^+$ or $\fM^-$.

\begin{prop}\label{P-7.9}
For every derived operator mean $M\in\frak{M}^+$,
\begin{align}\label{F-7.18}
{\det}^{1/N}M(\mu)\ge M(({\det}^{1/N})_*\mu),\qquad\mu\in\cPP,
\end{align}
where $({\det}^{1/N})_*\mu$ is the push-forward of $\mu$ by
$A\in\bP\mapsto\det^{1/N}A\in(0,\infty)$. The reversed inequality holds if $M\in\frak{M}^-$.
\end{prop}

\begin{proof}
The proof is again similar to that of Theorem \ref{T-6.5}. Here we only show that if
$M\in\fM$ satisfies \eqref{F-7.18} and $\sigma$ is a g.c.v.\ two-variable operator mean
with $\sigma\ne\frak{l}$, then the deformed $M_\sigma$ does the same. For every $\mu\in\cPP$
let $X_0:=M_\sigma(\mu)$. Then
\begin{align*}
{\det}^{1/N}X_0&={\det}^{1/N}M(X_0\sigma\mu)\ge M(({\det}^{1/N})_*(X_0\sigma\mu)) \\
&\ge M(({\det}^{1/N}X_0)\sigma(({\det}^{1/N})_*\mu)),
\end{align*}
where the last inequality can be shown by use of Lemma \ref{L-7.8} as in the proof of
Proposition \ref{P-7.2}. By Theorem \ref{T-3.1}\,(2) (for the case $\dim\cH=1$) this implies
that ${\det}^{1/N}X_0\ge M_\sigma(({\det}^{1/N})_*\mu)$.

The latter assertion follows from the first since
${\det}^{1/N}M(\mu^{-1})={\det}^{-1/N}M^*(\mu)$ and $M\in\frak{M}^+$ $\iff$
$M^*\in\frak{M}^-$.
\end{proof}

In particular, since $G\in\frak{M}^+\cap\frak{M}^-$, we have
$$
{\det}^{1/N}G(\mu)=G(({\det}^{1/N})_*\mu)=\exp\int_\bP{1\over N}\,\tr(\log A)\,d\mu(A)
$$
for every $\mu\in\cPP$, as verified in \cite{Li}. For power means Proposition \ref{P-7.9}
gives:

\begin{cor}
For every $\mu\in\cPP$,
\begin{align*}
&{\det}^{1/N}P_\alpha(\mu)\ge P_\alpha(({\det}^{1/N})_*\mu)
\quad\mbox{for $0<\alpha\le1$}, \\
&{\det}^{1/N}P_\alpha(\mu)\le P_\alpha(({\det}^{1/N})_*\mu)
\quad\mbox{for $-1\le \alpha<0$}.
\end{align*}
\end{cor}

In the case of the $n$-variable power means, the above inequality means that, for any
weight vector $\bw$ and for every $A_1,\dots,A_n\in\bP$,
\begin{align*}
&{\det}^{1/N}P_{\bw,\alpha}(A_1,\dots,A_n)
\ge P_{\bw,\alpha}({\det}^{1/N}A_1,\dots,{\det}^{1/N}A_n)
\quad\mbox{for $0<\alpha\le1$}, \\
&{\det}^{1/N}P_{\bw,\alpha}(A_1,\dots,A_n)
\ge P_{\bw,\alpha}({\det}^{1/N}A_1,\dots,{\det}^{1/N}A_n)
\quad\mbox{for $-1\le \alpha<0$}.
\end{align*}

\begin{problem}\rm
It seems interesting to extend \eqref{F-7.18} to the Fuglede-Kadison determinant in a finite
von Neumann algebra. Let $\cN$ be a von Neumann algebra on $\cH$ with a faithful normal
tracial state $\tau$. For every invertible operator $X\in\cN$ the
\emph{Fuglede-Kadison determinant} of $X$ is given
as
$$
\Delta(X):=\exp\tau(\log|X|).
$$
Let $M\in\frak{M}$ and $\mu\in\cPP$. Assume that $\mu$ is supported on $\bP(\cN)$, the set
of positive invertible operators in $\cN$. Then, for every unitary $U\in\cN'$, the commutant
of $\cN$, we have $UM(\mu)U^*=M(U\mu U^*)=M(\mu)$. This says that $M(\mu)\in\cN$ and so
$M(\mu)\in\bP(\cN)$. We may conjecture that the extension of Proposition \ref{P-7.9} holds as
\begin{align*}
&\Delta(M(\mu))\ge M(\Delta_*\mu)\quad\mbox{for $M\in\frak{M}^+$}, \\
&\Delta(M(\mu))\le M(\Delta_*\mu)\quad\mbox{for $M\in\frak{M}^-$}.
\end{align*}
\end{problem}

\section*{Acknowledgments}

The work of F.~Hiai was supported in part by JSPS KAKENHI Grant Number JP17K05266.
 The work of Y.~Lim was supported by the National
Research Foundation of Korea (NRF) grant funded by the Korea
government(MEST) No.2015R1A3A2031159 and 2016R1A5A1008055.

\appendix

\section{SOT-continuity of certain operator means}

The next proposition is concerned with the SOT-continuity of two-varialbe operator means.

\begin{prop}\label{P-A.1}
Any two-variable operator mean $\sigma$ $($in the Kubo-Ando sense$)$
is SOT-continuous on $\{A\in\bP:A\ge\eps I\}\times\bP$ and on
$\bP\times\{A\in\bP:A\ge\eps I\}$ for every $\eps>0$.
\end{prop}

\begin{proof}
Let $A,A_k,B,B_k\in\bP$ ($k\in\bN$), where $A,A_k\ge\eps I$ for some $\eps>0$, and assume
that $A_k\to A$ and $B_k\to B$ in SOT. Note that $A_k$'s and $B_k$'s are $\|\cdot\|$-bounded
by the uniform boundedness theorem. Then $A_k^{1/2}\to A_{1/2}$ and $A_k^{-1/2}\to A^{-1/2}$
in SOT, so that $A_k^{-1/2}B_kA_k^{-1/2}\to A^{-1/2}BA^{-1/2}$ in SOT. Therefore,
$f_\sigma(A_k^{-1/2}B_kA_k^{-1/2})\to f_\sigma(A^{-1/2}BA^{-1/2})$ in SOT, which implies that
$A_k\sigma B_k=A_k^{1/2}f_\sigma(A_k^{-1/2}B_kA_k^{-1/2})A_k^{1/2}\to A\sigma B$ in SOT.
Let $\sigma'$ be the \emph{transpose} of $\sigma$ \cite{KA}. We have
$B_k\sigma A_k=A_k\sigma'B_k\to A\sigma'B=B\sigma A$ in SOT.
\end{proof}

\begin{remark}\label{R-A.2}\rm
When $\dim\cH=\infty$, the operator mean $\sigma$ is not necessarily SOT-continuous on the
whole $\bP\times\bP$. For this, recall the well-known fact that there are
dense subspaces $\cK$ and $\cL$ of $\cH$ such that $\cK\cap\cL=\{0\}$. This follows from a
classical result of von Neumann, who proved that, for any unbounded self-adjoint operator
$T$, there exists a unitary operator $U$ such that the domains of $U$ and $U^*TU$ have the
zero intersection. A readable exposition on this matter is found in \cite{FW}. One can then
construct two orthonormal bases $\{e_j\}_{j=1}^\infty$ and $\{f_j\}_{j=1}^\infty$ of $\cH$
which are in $\cK$ and $\cL$, respectively. Let $P_k,Q_k$ ($k\in\bN$) be the orthogonal
projections onto the spans of $\{e_j\}_{j=1}^k$ and of $\{f_j\}_{j=1}^k$, respectively. Now
let $\sigma$ be, for instance, the geometric mean or the harmonic mean, and choose a unit
vector $\xi\in\cH$. For each $k$, since $P_k\sigma Q_k=P_k\wedge Q_k=0$ by \cite[(3.11)]{KA},
one can choose an $\eps_k>0$ such that $\<\xi,(P_k+\eps_kI)\sigma(Q_k+\eps_kI)\xi\><1/2$.
Here a sequence $\eps_k$ can be chosen as $\eps_k\searrow0$. Set $A_k:=P_k+\eps_kI$ and
$B_k:=Q_k+\eps_kI$. Then $A_k\to I$ and $B_k\to I$ in SOT, but
$\<\xi,(A_k\sigma B_k)\xi\><1/2$ for all $k$, so that $A_k\sigma B_k\not\to I$ in SOT.
Since $P_k\nearrow I$ and $Q_k\nearrow I$, it is also seen that $\sigma$ is not upward
continuous on $B(\cH)^+\times B(\cH)^+$.
\end{remark}

Next, to give a proof of the next proposition mentioned in Problem \ref{Q-5.4}, we recall
a few more basic facts on the Wasserstein distance $d_p^W$ on a complete metric space
$(X,d)$, in addition to an account in Section 2. Let $\cP_0(X)$ be the set of finitely
supported uniform probability measures on $X$, i.e., probability measures of the form
$\mu=(1/n)\sum_{i=1}^n\delta_{x_i}$ ($n\in\bN$, $x_i\in X$). It is well-known (see \cite{St})
that $\cP_0(X)$ is $d_p^W$-dense in $\cP^p(X)$ for any $p\in[1,\infty)$. When
$\mu,\nu\in\cP_0(X)$ where $\mu=(1/n)\sum_{i=1}^n\delta_{x_i}$ and
$\nu=(1/n)\sum_{i=1}^n\delta_{y_i}$, we have (see \cite[p.\ 5]{Vi})
\begin{align}\label{F-A.1}
d_p^W(\mu,\nu)=\min_{\sigma\in S_n}\Biggl[{1\over n}\sum_{i=1}^n
d^p(x_i,y_{\sigma(i)})\Biggr]^{1/p},\qquad1\le p<\infty,
\end{align}
where $S_n$ is the permutation group on $\{1,\dots,n\}$.

\begin{prop}\label{P-A.3}
If $\mu,\mu_k\in\cP^\infty(\bP)$ $(k\in\bN)$ are supported on
$\Sigma_\eps$ with $\eps\in(0,1)$ and $\mu_k\to\mu$ weakly on
$\Sigma_\eps$ with SOT, i.e., $\int_\bP f(A)\,d\mu_k(A)\to\int_\bP
f(A)\,d\mu(A)$ for every bounded SOT-continuous real function $f$ on
$\Sigma_\eps$, then $\cA(\mu_k)\to\cA(\mu)$ and
$\cH(\mu_k)\to\cH(\mu)$ in SOT.
\end{prop}

\begin{proof}
As mentioned before Definition \ref{D-2.1}, $\Sigma_\eps$ with SOT
is a Polish space with the metric $d_\eps$ in \eqref{F-2.1}. Let
$\cP(\Sigma_\eps)$ ($\subset\cPP$) be the set of Borel probability
measures supported on $\Sigma_\eps$, and consider the $1$-Wasserstein distance
$(d_\eps)_1^W$ on $\cP(\Sigma_\eps)$ with respect to $d_\eps$, as in \eqref{F-2.2} with
$d_\eps$ in place of $\dT$. Since $\sup\{d_\eps(A,B):A,B\in\Sigma_\eps\}<\infty$, it
follows from \cite[Theorem 7.12]{Vi} that $\mu_k\to\mu$ weakly on $(\Sigma_\eps,d_\eps)$
if and only if $(d_\eps)_1^W(\mu_k,\mu)\to0$. For $\mu=(1/N)\sum_{i=1}^N\delta_{A_i}$ and
$\nu=(1/N)\sum_{i=1}^N\delta_{B_i}$ in $\cP_0(\Sigma_\eps)$, we have by \eqref{F-A.1}
$$
(d_\eps)_1^W(\mu,\nu)=\min_{\sigma\in S_N}{1\over N}\sum_{i=1}^Nd_\eps(A_i,B_{\sigma(i)}).
$$
Since, for any $\sigma\in S_N$,
\begin{align*}
d_\eps(\cA(\mu),\cA(\nu))
&=\sum_{n=1}^\infty{1\over2^n}\Bigg\|\Biggl({1\over N}\sum_{i=1}^NA_i
-{1\over N}\sum_{i=1}^NB_{\sigma(i)}\Biggr)\Bigg\| \\
&\le{1\over N}\sum_{n=1}^\infty\sum_{i=1}^N{1\over2^n}\|(A_i-B_{\sigma(i)})x_n\|
={1\over N}\sum_{i=1}^Nd_\eps(A_i,B_{\sigma(i)}),
\end{align*}
we have
\begin{align}\label{F-A.2}
d_\eps(\cA(\mu),\cA(\nu))\le(d_\eps)_1^W(\mu,\nu),\qquad\mu,\nu\in\cP_0(\Sigma_\eps).
\end{align}
For any $\mu,\nu\in\cP(\Sigma_\eps)$ choose sequences $\mu_k,\nu_k\in\cP_0(\Sigma_\eps)$
such that $(d_\eps)_1^W(\mu_k,\mu)\to0$ and $(d_\eps)_1^W(\nu_k,\nu)\to0$. Since
$d_\eps(\cA(\mu_k),\cA(\mu_l))\le(d_\eps)_1^W(\mu_k,\mu_l)\to0$ as $k,l\to\infty$, it
follows that $\cA(\mu_k)$ converges to $A_0\in\Sigma_\eps$ in SOT. Since
$(d_\eps)_1^W(\mu_k,\mu)\to0$ implies that
$\int_{\Sigma_\eps}\<x,Ay\>\,d\mu_k(A)\to\int_{\Sigma_\eps}\<x,Ay\>\,d\mu(A)$,
i.e., $\<x,\cA(\mu_k)y\>\to\<x,\cA(\mu)y\>$ for any $x,y\in\cH$, so $\cA(\mu_k)\to\cA(\mu)$
in the weak operator topology. Hence we find that $A_0=\cA(\mu)$ so that
$\cA(\mu_k)\to\cA(\mu)$ in SOT, and similarly $\cA(\nu_k)\to\cA(\nu)$ in SOT. Therefore,
by applying \eqref{F-A.2} to $\mu_k,\nu_k$ and taking the limit, inequality \eqref{F-A.2}
can extend to
$$
d_\eps(\cA(\mu),\cA(\nu))\le(d_\eps)_1^W(\mu,\nu),\qquad\mu,\nu\in\cP(\Sigma_\eps),
$$
from which we have $\cA(\mu_k)\to\cA(\mu)$ in SOT if $\mu,\mu_k\in\cP(\Sigma_\eps)$ and
$\mu_k\to\mu$ weakly on $(\Sigma_\eps,d_\eps)$, i.e., $(d_\eps)_1^W(\mu_k,\mu)\to0$.

The assertion for $\cH$ immediately follows from that for $\cA$ since $A\mapsto A^{-1}$ is
a homeomorphic self-map on $\Sigma_\eps$.
\end{proof}

\section{Some proofs}

\subsection{Proof of \eqref{F-7.14}}

Let $A,B\in\bP$. The well-known Ky Fan majorization (see \cite{Hi,MOA}) says that
$$
\lambda(A+B)\prec\lambda(A)+\lambda(B).
$$
Replacing $A,B$ with $A^{-1},B^{-1}$ and applying the convex function $x^{-1}$ on
$(0,\infty)$, we have
$$
\lambda(A^{-1}+B^{-1})^{-1}\prec_w(\lambda(A^{-1})+\lambda(B^{-1}))^{-1}.
$$
Note that $\lambda(A^{-1}+B^{-1})^{-1}=\lambda^\uparrow(A:B)$ and
$\lambda(A^{-1})=\lambda^\uparrow(A)^{-1}$, where $A:B:=(A^{-1}+B^{-1})^{-1}$, the
\emph{parallel sum} of $A,B$ and $\lambda^\uparrow(A)$ is the eigenvalues of $A$ arranged in
increasing order. Therefore,
$$
\lambda^\uparrow(A:B)\prec_w\lambda^\uparrow(A):\lambda^\uparrow(B),
$$
Since $(\lambda^\uparrow(A):\lambda^\uparrow(B))^\downarrow=\lambda(A):\lambda(B)$, the above
means that
\begin{align}\label{F-B.1}
\lambda(A:B)\prec_w\lambda(A):\lambda(B).
\end{align}
Now consider the integral expression in \eqref{F-7.2}. Note that
\begin{align*}
{1+t\over t}\|(tA):B\}\|&\le{1+t\over t}\{(t\|A\|):\|B\|\}
={(1+t)\|A\|\,\|B\|\over t\|A\|+\|B\|} \\
&\le\max\{\|A\|,\|B\|\},\qquad t\in(0,\infty).
\end{align*}
Hence, by approximating the integral with Riemann sums, we have
$$
\int_{(0,\infty)}{1+t\over t}\{(tA):B\}\,dm(t)
=\lim_{n\to\infty}\sum_{j=1}^{n^2}
m\biggl(\biggl({j-1\over n},{j\over n}\biggr]\biggr){1+{j\over n}\over{j\over n}}
\biggl\{\biggl({j\over n}\,A\biggr):B\biggr\}
$$
in the operator norm, and similarly
\begin{align*}
&\int_{(0,\infty)}{1+t\over t}\{(t\lambda_i(A)):\lambda_i(B)\}\,dm(t) \\
&\qquad=\lim_{n\to\infty}\sum_{j=1}^{n^2}
m\biggl(\biggl({j-1\over n},{j\over n}\biggr]\biggr){1+{j\over n}\over{j\over n}}
\biggl\{\biggl({j\over n}\,\lambda_i(A)\biggr):\lambda_i(B)\biggr\},\qquad1\le i\le N.
\end{align*}
Hence, from \eqref{F-7.2} and the Ky Fan majorization, we have for $1\le k\le N$,
\begin{align*}
&\sum_{i=1}^k\lambda_i(A\sigma B) \\
&\quad\le\lim_{n\to\infty}\sum_{i=1}^k\Biggl[a\lambda_i(A)+b\lambda_i(B)+\sum_{j=1}^{n^2}
m\biggl(\biggl({j-1\over n},{j\over n}\biggr]\biggr){1+{j\over n}\over{j\over n}}
\lambda_i\biggl(\biggl({j\over n}\,A\biggr):B\biggr)\Biggr] \\
&\quad\le\lim_{n\to\infty}\sum_{i=1}^k\Biggl[a\lambda_i(A)+b\lambda_i(B)+\sum_{j=1}^{n^2}
m\biggl(\biggl({j-1\over n},{j\over n}\biggr]\biggr){1+{j\over n}\over{j\over n}}
\biggl\{\biggl({j\over n}\,\lambda_i(A)\biggr):\lambda_i(B)\biggr\}\Biggr] \\
&\quad=\sum_{i=1}^k\{\lambda_i(A)\sigma\lambda_i(B)\},
\end{align*}
where we have used \eqref{F-B.1} for the above latter inequality. Therefore,
$\lambda(A\sigma B)\prec_w\lambda(A)\sigma\lambda(B)$.\qed

\subsection{Proofs of \eqref{F-7.12} and \eqref{F-7.13}}

Let $\mu\in\cPP$ and choose an $\eps>0$ such that $\mu$ is supported on $\Sigma_\eps$. Since
$\Sigma_\eps$ is compact in the operator norm thanks to $\dim\cH<\infty$, one can choose,
for any $n\in\bN$, a finite set $\{A_1^{(n)},\dots,A_{m_n }^{(n)}\}$ in $\bP$ such that
$\Sigma_\eps\subset\bigcup_{j=1}^{m_n}\{A\in\bP:\|A-A_j^{(n)}\|<1/n\}$. So one can define
disjoint Borel sets $\cO_j^{(n)}$ ($1\le j\le m_n$) such that
$\cO_j^{(n)}\subset\{A:\|A-A_j^{(n)}\|<1/n\}$ for $1\le j\le m_n$ and
$\Sigma_\eps=\bigcup_{j=1}^{m_n}\cO_j^{(n)}$. Then it is obvious that
$$
\lim_{n\to\infty}\bigg\|\cA(\mu)-\sum_{j=1}^{m_n}\mu(\cO_j^{(n)})A_j^{(n)}\bigg\|=0,
\quad
\lim_{n\to\infty}\bigg\|\cH(\mu)-\Biggl[\sum_{j=1}^{m_n}\mu(\cO_j^{(n)})A_j^{(n)-1}
\Biggr]^{-1}\bigg\|=0.
$$
Moreover, for every $i=1,\dots,N$,
$$
\int_\bP\lambda_i(A)\,d\mu(A)
=\lim_{n\to\infty}\sum_{j=1}^{m_n}\mu(\cO_j^{(n)})\lambda_i(A_j^{(n)}),
$$
$$
\biggl[\int_\bP\lambda_i(A)^{-1}\,d\mu(A)\biggr]^{-1}
=\lim_{n\to\infty}\Biggl[\sum_{j=1}^{m_n}
\mu(\cO_j^{(n)})\lambda_i(A_j^{(n)})^{-1}\Biggr]^{-1}.
$$
Note that
$$
\cA(\lambda_*\mu)=\biggl(\int_\bP\lambda_i(A)\,d\mu(A)\biggr)_{i=1}^N,\quad
\cH(\lambda_*\mu)=\biggl(\biggl[
\int_\bP\lambda_i(A)^{-1}\,d\mu(A)\biggr]^{-1}\biggr)_{i=1}^N.
$$
By the Ky Fan majorization we have
$$
\lambda\Biggl(\sum_{j=1}^{m_n}\mu(\cO_j^{(n)})A_j^{(n)}\Biggr)
\prec\sum_{j=1}^{m_n}\mu(\cO_j^{(n)})\lambda(A_j^{(n)}),
$$
$$
\lambda\Biggl(\Biggl[\sum_{j=1}^{m_n}\mu(\cO_j^{(n)})A_j^{(n)-1}\Biggr]^{-1}\Biggr)
\prec_w\Biggl[\sum_{j=1}^{m_n}\mu(\cO_j^{(n)})\lambda(A_j^{(n)})^{-1}\Biggr]^{-1}.
$$
where the last weak majorization is verified similarly to \eqref{F-B.1}. Letting $n\to\infty$
gives $\cA(\mu)\prec\cA(\lambda_*\mu)$ and $\cH(\mu)\prec_w\cH(\lambda_*\mu)$.\qed

\end{document}